\documentclass{article}
\usepackage{amsmath,amsthm,amsfonts,graphicx}
\usepackage{multirow}
\usepackage{slashbox}
\usepackage{multirow}
\def\smallddots{\mathinner{\raise7pt\hbox{.}\raise4pt\hbox{.}\raise1pt\hbox{.}}}
\def\smallsdots{\mathinner{\raise1pt\hbox{.}\raise4pt\hbox{.}\raise7pt\hbox{.}}}

\DeclareMathOperator{\diag}{diag}

\DeclareMathOperator{\rank}{rank}
\DeclareMathOperator{\nrank}{nrank}

\numberwithin{equation}{section}
\numberwithin{table}{section}
\newtheorem{theorem}{Theorem}[section]
\newtheorem{lemma}{Lemma}[section]

\newtheorem{corollary}{Corollary}[section]
\newtheorem{fact}{Fact}[section]

\newtheorem{protoalgorithm}{Proto-Algorithm}[section]

\newtheorem{definition}{Definition}[section]

\newtheorem{remark}{Remark}[section]

\begin{document}

\title{\bf More on the Power of Randomized Matrix Multiplication
\thanks {Some results of this paper have been presented at the 
 ACM-SIGSAM International 
Symposium on Symbolic and Algebraic Computation (ISSAC '2011), San Jose, CA, 2011,
the
3nd International Conference on Matrix Methods in Mathematics and 
Applications (MMMA 2011) in
Moscow, Russia, June 22-25, 2011, 
the 7th International Congress on Industrial and Applied Mathematics 
(ICIAM 2011), in Vancouver, British Columbia, Canada, July 18-22, 2011,
the SIAM International Conference on Linear Algebra,
in Valencia, Spain, June 18-22, 2012, and 
the Conference on Structured Linear and Multilinear Algebra Problems (SLA2012),
in  Leuven, Belgium, September 10-14, 2012}}

\author{Victor Y. Pan$^{[1, 2],[a]}$ and Guoliang Qian$^{[2],[b]}$
\and\\
$^{[1]}$ Department of Mathematics and Computer Science \\
Lehman College of the City University of New York \\
Bronx, NY 10468 USA \\
$^{[2]}$ Ph.D. Programs in Mathematics  and Computer Science \\
The Graduate Center of the City University of New York \\
New York, NY 10036 USA \\
$^{[a]}$ victor.pan@lehman.cuny.edu \\
http://comet.lehman.cuny.edu/vpan/  \\
$^{[b]}$ gqian@gc.cuny.edu \\
} 
 \date{}

\maketitle


\begin{abstract}
A random matrix is likely to be well 
conditioned, and motivated by
this well known property 
 we employ
random matrix multipliers
to advance some fundamental
 matrix 
computations. This includes
numerical stabilization of
  Gaussian 
elimination with no pivoting
as well as block Gaussian elimination, 
 approximation of the 
leading and trailing singular spaces
of an ill conditioned matrix, 
 associated with 
its largest and smallest
singular values, respectively,
and approximation of this matrix 
by 
low-rank matrices,
with further extensions to 
Tensor Train approximation  
and the computation of 
the numerical rank of a matrix.
We formally support 
the efficiency of the 
proposed techniques where 
we employ Gaussian random 
multipliers, but our extensive tests
have consistently
produced 
the same outcome 
where instead we 
used sparse
and 
structured random multipliers, 
defined by much fewer random parameters
compared to the number of their entries.
\end{abstract}

\paragraph{\bf 2000 Math. Subject Classification:}
 15A52, 15A12, 15A06, 65F22, 65F05

\paragraph{\bf Key Words:}
Random matrices,
Random multipliers,
GENP,
Low-rank approximation,
Numerical rank
 


\section{Introduction}\label{sintro}


It is well known that A random matrix is likely
to be well conditioned
 \cite{D88},  \cite{E88},  \cite{ES05},
\cite{CD05}, \cite{SST06}, \cite{B11}, and  
 motivated by
this well known property we 
apply randomized matrix multiplication
to advance some fundamental 
matrix computations. 
We stabilize
 numerically Gaussian 
elimination with no pivoting
as well as block Gaussian elimination,  
approximate leading and trailing singular spaces
of an ill conditioned matrix $A$,
 associated with 
its largest and smallest singular values, 
respectively,
approximate this matrix by low-rank matrices,
and compute Tensor Train 
 approximation,
the numerical rank of a matrix, and an 
approximation of a matrix by a structured 
matrix lying  nearby.
Our numerical tests  
are in good accordance with our formal study,
except that in the tests
all algorithms  have fully preserved their 
 power even
where we dramatically decreased the
number of random parameters involved
by using sparse and structured multipliers.


\subsection{Numerically safe Gaussian elimination with no pivoting}\label{sgenp}


Hereafter ``flop" stands for ``arithmetic operation",
by saying ``expect" and ``likely" we
mean ``with probability $1$ or close to $1$",
$\sigma_j(A)$ denotes the $j$th largest
singular value
of an $n\times n$ matrix $A$, 
and the ratio
 $\kappa (A)=\sigma_1(A)/\sigma_{\rho}(A)$
for $\rho=\rank (A)$
 denotes its condition number. 
$\kappa (A)=||A||~||A^{-1}||$ if 
$\rho=n$, that  is if $A$ is a nonsingular matrix.
If this number is large in context,
then the matrix $A$
is {\em ill conditioned}, otherwise {\em well conditioned}.
For matrix inversion and solving linear systems of equations
the condition number represents the output magnification 
of 
input errors,
\begin{equation}\label{eqkappa}
\kappa(A)\approx \frac{||{\rm OUTPUT~ERROR}||}{||{\rm INPUT~ERROR}||},
\end{equation}
and backward error analysis implies similar 
magnification of rounding errors
\cite{GL96}, \cite{H02}, \cite{S98}.

To avoid dealing with singular or ill conditioned matrices in
Gaussian elimination, one incorporates pivoting, that is row or column interchange.
 {\em Gaussian elimination with no pivoting} (hereafter 
we refer to it as {\em GENP}) can easily fail in numerical 
computations with rounding errors, except for the cases
where the input
matrices are strongly well conditioned,
that is where all their leading principal square blocks
are nonsingular and  well conditioned.
 In particular 
diagonally 
dominant as well as positive definite 
well conditioned matrices have this property. 
For such matrices,
GENP outperforms
 Gaussian elimination with pivoting
\cite[page 119]{GL96}.
Random matrices are likely to be 
strongly well conditioned, but
we do not solve random linear systems of equations.
We can, however randomize linear systems
by applying random multipliers
and then can apply GENP. We proposed and tested
this approach 
in \cite[Section 12.2]{PGMQ} 
and \cite{PQZa},
 and our tests consistently
showed its efficiency even where
we used just 
  circulant or Householder multipliers  
filled with integers $\pm 1$ 
and where we
limited randomization to the choice
of the signs $\pm$
(see our Table \ref{tab44}
and \cite[Table 2]{PQZa}).
Our Corollary \ref{cogh} supports these empirical 
observations provided that the multipliers 
are square 
Gaussian random matrices.
Estimation of the condition numbers of 
structured matrices
was stated as a challenge in \cite{SST06}.
The paper \cite{PQa} presents some initial
advance, but the problem remains largely open.


\subsection{Randomized low-rank approximation and beyond}\label{srdpr}


Our Corollary \ref{cogh}, 
supporting randomized GENP, relies on the probabilistic estimates
for the ranks and condition numbers of the products $\kappa (GA)$ and $\kappa (AH)$ 
in terms of $\kappa (A)$ where $G$ and $H$ are Gaussian random matrices
(see Theorem \ref{1}). We also apply the same estimates to support
randomized algorithms for the approximation of the
 leading singular spaces
of an ill conditioned matrix $A$
 associated with 
its largest singular values. This can be 
immediately extended to the
approximation of a matrix having a small numerical rank by 
low-rank matrices. The 
algorithm is numerially safe, runs at a low 
computational cost, and has a great number of
highly important applications to
matrix computations \cite{HMT11}.
We point out its further extensions 
to  the
approximation of a matrix by a structured 
matrix lying  nearby
and to computing Tensor Train 
 approximation and the numerical rank of a matrix. 
Then again our formal support of these algorithms relies on 
using Gaussian random multipliers,
but our tests show  
that random Toeplitz multipliers 
are as effective. This suggests
formal and experimantal study
of various other random structured and sparse
multipliers that depend on smaller
numbers of random parameters.
Note the recent success of Tropp \cite{T11}
in this direction.

\subsection{Related work}\label{srel}

Preconditioning  of linear systems of equations  
is a classical subject \cite{A94}, \cite{B02}, \cite{G97}.
Randomized multiplicative preconditioning
for numerical stabilization of GENP 
was proposed in \cite[Section 12.2]{PGMQ} 
and \cite{PQZa}, but
with no formal support for this approach.
On low-rank approximation
we refer the reader to the survey \cite{HMT11}.
We cite these and other related works throughout 
the paper and refer to \cite[Section 11]{PQZb}
on further bibliography. For a natural extension
of our present work,
 one can combine
randomized matrix multiplication with
randomized augmentation and 
 additive preprocessing
of \cite{PGMQ},
 \cite{PIMR10},  \cite{PQ10}, \cite{PQ12},
 \cite{PQZC}, \cite{PQZb}, \cite{PY09}.
 

\subsection{Organization of the paper and selective reading}\label{sorg}


In the next 
section
we recall some definitions and basic results.
We  estimate  
the condition numbers of Gaussian random 
matrices
in Section \ref{srrm} and 
of
randomized  matrix products
in Section \ref{smrc},
where we also 
comment on numerical stabilization of GENP 
by means of randomized multilication.
In Sections \ref{sapsr} and \ref{svianvtns}
 we 
apply randomized matrix multiplication to
approximate the leading and trailing singular spaces of a matrix having 
a small numerical rank,
approximate this matrix by a low-rank matrix,
and point out applications to tensor decomposition
and to approximation by structured matrices.
 In Section \ref{sexp} we cover numerical tests,
which constitute the contribution of the second author.
In Appendix A we estimate the probability that 
 a 
 random matrix 
has full rank
under the uniform probability distribution.
In Appendix B we compute the numerical 
rank of a matrix by
using randomization but neither pivoting nor orthogonaliztaion.


\section{Some definitions and basic results}\label{sdef}


We assume computations in the field $\mathbb R$ of real numbers.

Hereafter ``flop" stands for ``arithmetic operation"; 
``expect" and ``likely"
mean ``with probability $1$ or close to $1$"
(we do not use the concept of the expected value), and
 the concepts ``large", ``small", ``near", ``closely approximate", 
``ill  conditioned" and ``well conditioned" are 
quantified in the context. 
Next we recall and extend some customary definitions of matrix computations
\cite{GL96}, \cite{S98}.


\subsection{Some basic definitions on matrix computations}\label{smat}

 

$\mathbb  R^{m\times n}$ is the class of real $m\times n$ matrices $A=(a_{i,j})_{i,j}^{m,n}$.

$(B_1~|~\dots~|~B_k)=(B_j)_{j=1}^k$ is a $1\times k$ block matrix with blocks $B_1,\dots,B_k$. 
$\diag (B_1,\dots,B_k)=\diag(B_j)_{j=1}^k$ is a $k\times k$ block diagonal matrix with diagonal blocks $B_1,\dots,B_k$.

${\bf e}_i$ is the $i$th coordinate vector of dimension $n$ for
$i=1,\dots,n$. These vectors define
the identity
matrix $I_n=({\bf e}_1~|~\dots~|~{\bf e}_n)$ 
of size $n\times n$. 
$O_{k,l}$ is the $k\times l$ matrix filled with zeros. 
We write $I$ and $O$ where the size of a matrix 
is not important or is defined by context. 

$A^T$ is the transpose  of a 
matrix $A$. 


\subsection{Range, rank, and
generic rank profile}\label{srnsn}


$\mathcal R(A)$  denotes the range of an $m\times n$ matrix $A$, that is the linear space $\{{\bf z}:~{\bf z}=A{\bf x}\}$
generated by its columns. 
$\rank (A)=\dim \mathcal R(A)$ denotes its rank.
$A_k^{(k)}$ denotes the leading, 
that is northwestern  $k\times k$ block 
submatrix of a matrix $A$.
A matrix of a rank $\rho$ has {\em generic rank profile}
if all its leading $i\times i$ blocks are nonsingular for $i=1,\dots,\rho$.
If such matrix is nonsingular itself, then  it is called {\em strongly nonsingular}. 


\begin{fact}\label{far1}
The set $\mathbb M$ of $m\times n$ matrices of rank  $\rho$
is an algebraic variety of dimension  $(m+n-\rho)\rho$.
\end{fact}
\begin{proof}
Let $M$ be an $m\times n$ matrix of a rank $\rho$
with a nonsingular leading $\rho\times \rho$ block $M_{00}$
and write $M=\begin{pmatrix}
M_{00}   &   M_{01}   \\
M_{10}   &   M_{11}
\end{pmatrix}$.
Then the $(m- \rho)\times (n- \rho)$ {\em Schur complement} $M_{11}-M_{10}M_{00}^{-1}M_{01}$
must vanish, which imposes $(m-\rho)(n-\rho)$ algebraic equations on the entries of $M$. 
Similar argument can be applied  where any $\rho\times \rho$
submatrix of the matrix $M$ 
(among $\begin{pmatrix}
m      \\
\rho
\end{pmatrix}\begin{pmatrix}
n     \\
\rho
\end{pmatrix}$ such submatrices)
is nonsingular. Therefore 
$\dim \mathbb M=mn-(m-\rho)(n-\rho)=(m+n-\rho)\rho$.
\end{proof}  




\subsection{Orthogonal, Toeplitz and circulant matrices}\label{smatotc}

A real matrix $Q$ is called  
{\em orthogonal} if $Q^TQ=I$ 
 or $QQ^T=I$. 
In Section \ref{sexp} we write $Q(A)$ to denote a unique orthogonal matrix 
specified by the following result.

\begin{fact}\label{faqrf} \cite[Theorem 5.2.2]{GL96}.
QR factorization $A=QR$ of a matrix $A$ having full column rank
into the product of an orthogonal matrix $Q=Q(A)$ 
and an upper triangular matrix $R=R(A)$ is unique 
provided that the factor $R$ is a square matrix 
with positive diagonal entries. 
\end{fact}


A {\em Toep\-litz} $m\times n$ matrix $T_{m,n}=(t_{i-j})_{i,j=1}^{m,n}$ 
is defined by its first row $(t_{-h})_{h=0}^{n-1}$
and the subvector  $(t_h)_{h=1}^{n-1}$ of its first column vector.
Circulant matrices are
the subclass of Toep\-litz 
 matrices where $t_g=t_h$ if $|g-h|=n$.  

\begin{theorem}\label{tht}
$O((m+n)\log (m+n))$ flops suffice to
multiply an $m\times n$ Toeplitz matrix by a vector. 
\end{theorem}


\subsection{Norms, SVD, generalized inverse, and singular spaces}\label{sosvdi}


$||A||_h$ is the $h$-norm and
$||A||_F=\sqrt{\sum_{i,j=1}^{m,n}|a_{i,j}|^2}$ is the Frobenius norm
of a matrix $A=(a_{i,j})_{i,j=1}^{m,n}$.
We write $||A||=||A||_2$ and $||{\bf v}||=\sqrt {{\bf v}^T{\bf v}}=||{\bf v}||_2$  
and recall from \cite[Section 2.3.2 and Corollary 2.3.2]{GL96} that
$$
{\rm max}_{i,j=1}^{m,n}|a_{i,j}|\le ||A||=||A^T||\le \sqrt {mn}~{\rm max}_{i,j=1}^{m,n}|a_{i,j}|,
$$
\begin{equation}\label{eqnorm12}
\frac{1}{\sqrt m}||A||_1\le||A||\le \sqrt n ||A||_1,~~||A||_1=||A^T||_{\infty},~~
||A||^2\le||A||_1||A||_{\infty}, 
\end{equation}
\begin{equation}\label{eqfrob}
||A||\le||A||_F\le \sqrt n ~||A||, 
\end{equation}
\begin{equation}\label{eqnorm12inf}
||AB||_h\le ||A||_h||B||_h~{\rm for}~h=1,2,\infty~{\rm and~any~matrix~product}~AB.
\end{equation}



Define an {\em SVD} or {\em full SVD} of an $m\times n$ matrix $A$ of a rank 
 $\rho$ as follows,
\begin{equation}\label{eqsvd}
A=S_A\Sigma_AT_A^T.
\end{equation}
Here
$S_AS_A^T=S_A^TS_A=I_m$, $T_AT_A^T=T_A^TT_A=I_n$,
$\Sigma_A=\diag(\widehat \Sigma_A,O_{m-\rho,n-\rho})$, 
$\widehat \Sigma_A=\diag(\sigma_j(A))_{j=1}^{\rho}$,
$\sigma_j=\sigma_j(A)=\sigma_j(A^T)$
is the $j$th largest singular value of a matrix $A$
 for $j=1,\dots,\rho$, and we write
$\sigma_j=0$ for $j>\rho$. These values have 
the minimax property  
\begin{equation}\label{eqminmax}
\sigma_j=\max_{{\rm dim} (\mathbb S)=j}~~\min_{{\bf x}\in \mathbb S,~||{\bf x}||=1}~~~||A{\bf x}||,~j=1,\dots,\rho,
\end{equation}
where $\mathbb S$ denotes linear spaces  \cite[Theorem 8.6.1]{GL96}.
Consequently 
 $\sigma_{\rho}>0$,  
 $\sigma_1=\max_{||{\bf x}||=1}||A{\bf x}||=||A||$.

 
\begin{fact}\label{faccondsub} 
If $A_0$ is a 
submatrix of a 
matrix $A$, 
then
$\sigma_{j} (A)\ge \sigma_{j} (A_0)$ for all $j$.
\end{fact} 

 
\begin{proof}
\cite[Corollary 8.6.3]{GL96} implies 
the claimed bound
where $A_0$ is any block of columns of 
the matrix $A$. Transposition of a matrix and permutations 
of its rows and columns do not change singular values,
and thus we can extend the bounds to
all submatrices $A_0$.
\end{proof}

$A^+=T_A\diag(\widehat \Sigma_A^{-1},O_{n-\rho,m-\rho})S_A^T$ is the Moore--Penrose 
pseudo-inverse of the matrix $A$ of (\ref{eqsvd}), and
\begin{equation}\label{eqnrm+}
||A^+||=1/\sigma_{\rho}(A)
\end{equation}
 for 
a matrix $A$ of a rank $\rho$.  $A^{+T}$ stands for $(A^+)^T=(A^T)^+$,
and $A^{-T}$ stands for $(A^{-1})^T=(A^T)^{-1}$.

In Sections \ref{sapsr}--\ref{svianvtns}
we use the following definitions.
For every integer $k$ in the range $1\le k<\rank(A)$ define the partition 
$S_A=(S_{k,A}~|~S_{A,m-k})$ and $T_A=(T_{k,A}~|~T_{A,n-k})$
where the submatrices $S_{k,A}$ and $T_{k,A}$ are formed by the
first $k$ columns of the matrices $S_A$ and $T_A$, respectively.
Write $\Sigma_{k,A}=\diag(\sigma_j(A))_{j=1}^k$,
$\mathbb S_{k,A}=\mathcal R(S_{k,A})$ and  $\mathbb T_{k,A}=\mathcal R(T_{k,A}$).
If $\sigma_k>\sigma_{k+1}$, 
then 
$\mathbb S_{k,A}$ and $\mathbb T_{k,A}$ are
the left and right {\em leading singular spaces}, respectively,
    associated with the $k$ largest singular values of the matrix $A$,
whereas their orthogonal complements $\mathbb S_{A,m-k}=\mathcal R(S_{A,m-k})$ 
and $\mathbb T_{A,n-k}=\mathcal R(T_{A,n-k})$ 
are the left and right {\em trailing singular spaces}, respectively,
associated with the other singular values of $A$. 
The pairs of subscripts $\{k,A\}$ versus $\{A,m-k\}$ and $\{A,n-k\}$ mark 
the leading versus trailing
singular spaces.     
The left singular spaces of $A$ are 
the right   singular spaces of $A^T$ and vice versa.
All matrix bases for the singular spaces $\mathbb S_{k,A}$ and $\mathbb T_{k,A}$
are given by matrices $S_{k,A}X$ and
$T_{k,A}Y$, respectively,
for  nonsingular $k\times k$ matrices $X$ and $Y$.
Orthogonal matrices $X$ and $Y$ define orthogonal matrix bases
for these spaces.  
$B$ is an {\em approximate matrix basis} for 
a space $\mathbb S$ within a relative error norm bound $\tau$
if there exists a matrix $E$ such that $B+E$ is a matrix basis for 
this space $\mathbb S$ and if $||E||\le \tau ||B||$.


\subsection{Condition number, numerical rank and  generic conditioning profile
}\label{scnpn}


$\kappa (A)=\frac{\sigma_1(A)}{\sigma_{\rho}(A)}=||A||~||A^+||$ is the condition 
number of an $m\times n$ matrix $A$ of a rank $\rho$. Such matrix is {\em ill conditioned} 
if $\sigma_1(A)\gg\sigma_{\rho}(A)$ and is {\em well conditioned}
otherwise. See \cite{D83}, \cite[Sections 2.3.2, 2.3.3, 3.5.4, 12.5]{GL96}, 
\cite[Chapter 15]{H02}, \cite{KL94}, \cite[Section 5.3]{S98}, 
on the estimation of matrix norms and condition numbers. 

An $m\times n$ matrix $A$ has {\em numerical rank}, denoted  $\nrank(A)$
and not exceeding $\rank (A)$, 
if the ratios $\sigma_{j}(A)/||A||$
are small for $j>\nrank(A)$ but not for $j\le \nrank(A)$. 


\begin{remark}\label{renul}
One can specify the adjective ``small" 
above as
``smaller than  a fixed positive tolerance".
The choice of the tolerance can be a challenge,
e.g., for the matrix $\diag(1.1^{-j})_{j=0}^{999}$.
\end{remark}

  
If a well conditioned  $m\times n$ matrix $A$ has a rank $\rho<l=\min\{m,n\}$, 
 then almost all its close neighbours have full rank $l$
(see Section \ref{sngrm}), and
all of them have numerical rank $\rho$.
Conversely, suppose a matrix $A$
has a positive
numerical rank $\rho=\nrank (A)$ and 
{\em truncate its SVD}  by
setting 
to $0$ all its 
singular values, except for the  $\rho$ largest  ones.
Then the resulting matrix  $A-E$ is well conditioned and
has rank $\rho$ and
 $||E||=\sigma_{\rho+1}(A)$,
and so $A-E$ is a rank-$\rho$ approximation to the matrix $A$
within the error norm bound $\sigma_{\rho+1}(A)$.
At a
lower computational cost we can obtain rank-$\rho$ approximations 
of  the matrix $A$ from its
  rank-revealing factorizations 
\cite{GE96}, \cite{HP92}, \cite{P00a}, 
and we further decrease the computational cost
by
applying randomized algorithms
in Section \ref{sapsr}. 




An $m\times n$ matrix 
has {\em generic conditioning profile}
(cf. the  end of Section \ref{srnsn})
if it
has a numerical rank $\rho$ and if
its leading $i\times i$ blocks are nonsingular and well conditioned for $i=1,\dots,\rho$.
If such matrix has full rank (that is if $\rho=\min\{m,n\}$) 
and if it is well conditioned  itself, 
 then we call it {\em strongly well conditioned}.
The following theorem 
shows that
GENP 
and block Gaussian elimination
applied to a strongly 
well conditioned matrix
are numerically safe.  


\begin{theorem}\label{thnorms} Cf.  \cite[Theorem 5.1]{PQZa}.
Assume GENP
or block Gaussian elimination
applied to 
 an
$n\times n$
matrix $A$ and
write $N=||A||$ and $N_-=\max_{j=1}^n ||(A_j^{(j)})^{-1}||$. 
Then the absolute values of all pivot elements of GENP 
and the norms of all pivot blocks of 
block Gaussian elimination
do not exceed $N+N_-N^2$,
whereas the absolute values of the reciprocals of these 
elements and the norms of the inverses of the blocks do not
exceed $N_-$.
\end{theorem}


\section{Ranks and conditioning of Gaussian random
matrices}\label{srrm}


\subsection{Random variables and  Gaussian random matrices}\label{srvrm}


\begin{definition}\label{defcdf}
$F_{\gamma}(y)=$ Probability$\{\gamma\le y\}$ (for a real random variable $\gamma$)
is the {\em cumulative 
distribution function (cdf)} of $\gamma$ evaluated at $y$. 
$F_{g(\mu,\sigma)}(y)=\frac{1}{\sigma\sqrt {2\pi}}\int_{-\infty}^y \exp (-\frac{(x-\mu)^2}{2\sigma^2}) dx$ 
for a Gaussian random variable $g(\mu,\sigma)$ with a mean $\mu$ and a positive variance $\sigma^2$,
and so   
\begin{equation}\label{eqnormal}
\mu-4\sigma\le y \le \mu+4\sigma~{\rm with ~a ~probability ~near ~1}.
\end{equation}
\end{definition}


\begin{definition}\label{defrndm}
A matrix (or a vector) is a {\em Gaussian random matrix (or vector)} with a mean 
$\mu$ and a positive variance $\sigma^2$ if it is filled with 
independent identically distributed Gaussian random 
variables, all having the mean $\mu$ and variance $\sigma^2$. 
$\mathcal G_{\mu,\sigma}^{m\times n}$ is the set of such
Gaussian  random  $m\times n$ matrices,
which are {\em standard} for $\mu=0$
and $\sigma^2=1$. By restricting this set 
to $m\times n$ Toeplitz matrices 
where only the $m+n-1$ entries of the first row and column
 are independent
we obtain the set of
$\mathcal T_{\mu,\sigma}^{m\times n}$ 
{\em Gaussian random Toep\-litz}  matrices .
Likewise we obtain the set 
$\mathcal Z_{\mu,\sigma}^{n\times n}$ of
{\em Gaussian random circulant matrices},
where only  $n$ entries of the first row are independent.   
\end{definition}




\subsection{Nondegeneration of Gaussian random matrices}\label{sngrm}


The total degree of a multivariate monomial is the sum of its degrees
in all its variables. The total degree of a polynomial is the maximal total degree of 
its monomials.


\begin{lemma}\label{ledl} \cite{DL78}, \cite{S80}, \cite{Z79}.
For a set $\Delta$ of a cardinality $|\Delta|$ in any fixed ring  
let a polynomial in $m$ variables have a total degree $d$ and let it not vanish 
identically on this set. Then the polynomial vanishes in at most 
$d|\Delta|^{m-1}$ points. 
\end{lemma}


We assume that Gaussian random variables range 
over infinite sets $\Delta$,
usually over the real line or its interval. Then
the lemma implies that a nonzero polynomial vanishes with probability 0.
Consequently a  Gaussian random general, Toeplitz or circulant
matrix has generic rank profile 
with probability 1
because the determinant
of any its block is a polynomials
in the entries. 
Likewise
 Gaussian random general, Toeplitz and circulant 
matrices have generic rank profile with probability 1.  Hereafter,
wherever this causes no confusion,  
we assume by default that
{\em Gaussian random 
 general, Toeplitz and circulant 
matrices  have generic rank profile.}
This property can be readily extended to the
products and various 
functions of general, sparse and structured Gaussian random matrices. 
Similar properties hold with probability near 1 
where the random variables are sampled
under the uniform probability distribution
from a finite set of a large cardinality 
(see the Appendix).
 

\subsection{Extremal singular values of Gaussian random matrices}\label{scgrm}


Besides having full rank with probability 1,
Gaussian random matrices in Definition \ref{defrndm} are  likely to be well conditioned  
\cite{D88}, \cite{E88}, \cite{ES05}, \cite{CD05}, \cite{B11}, and 
even the sum $M+A$ for  $M\in \mathbb R^{m\times n}$ and 
 $A\in \mathcal G_{\mu,\sigma}^{m\times n}$ is  likely to
be well conditioned unless the ratio  $\sigma/||M||$ is small
or large 
\cite{SST06}. 

The following theorem 
states an upper bound 
proportional to $y$ on
 the cdf $F_{1/||A^+||}(y)$, that is  
on the probability that  the 
smallest positive singular value $1/||A^+||=\sigma_l(A)$ of a  Gaussian random matrix $A$ 
is less than a nonnegative scalar $y$ (cf. (\ref{eqnrm+}))
and consequently on the probability that the norm $||A^+||$
exceeds a positive scalar $x$.
The stated bound still holds if we replace the matrix $A$ by 
$A-B$ for any fixed matrix $B$, and
for $B=O_{m,n}$
the  bounds
can  be strengthened  
by a factor $y^{|m-n|}$ \cite{ES05}, \cite{CD05}.


\begin{theorem}\label{thsiguna} 
Suppose 
$A\in \mathcal G_{\mu,\sigma}^{m\times n}$, 
 $B\in \mathbb R^{m\times n}$,
$l=\min\{m,n\}$,  $x>0$, and $y\ge 0$. 
Then 
$F_{\sigma_l(A-B)}(y)\le 2.35~\sqrt l y/\sigma$, 
that is
$Probability \{||(A-B)^+||\ge 2.35x\sqrt {l}/\sigma\}\le 1/x$.
\end{theorem}
\begin{proof}
For $m=n$ this is \cite[Theorem 3.3]{SST06}. Apply
 Fact \ref{faccondsub} to extend it to any pair $\{m,n\}$.
\end{proof}


The following two theorems supply lower bounds
$F_{||A||}(z)$ and
$F_{\kappa (A)}(y)$
 on the probabilities 
that $||A||\le z$ 
and $\kappa(A)\le y$ for two scalars $y$ and $z$, 
respectively,
and a Gaussian random matrix $A$. 
We do not use the second theorem, but state it for the sake of completeness
and only for square $n\times n$ matrices $A$.
The theorems  
imply that  
the functions 
$1-F_{||A||}(z)$
and
$1-F_{\kappa (A)}(y)$ 
decay as 
$z\rightarrow \infty$ and
$y\rightarrow \infty$, respectively,
and that the decays are exponential in $-z^2$ and  proportional 
to $\sqrt{\log y}/y$, respectively.
 For small values $y\sigma$ and a fixed $n$ 
the lower bound of Theorem \ref{thmsiguna}
becomes negative, in which case 
the theorem becomes trivial. 
Unlike Theorem \ref{thsiguna}, in both theorems we assume that $\mu=0$. 


\begin{theorem}\label{thsignorm} \cite[Theorem II.7]{DS01}.
Suppose $A\in \mathcal G_{0,\sigma}^{m\times n}$,
$h=\max\{m,n\}$  and
$z\ge 2\sigma\sqrt {h}$. 
Then $F_{||A||}(z)\ge 1- \exp(-(z-2\sigma\sqrt {h})^2/(2\sigma^2))$, and so
the norm $||A||$ is  likely to have order $\sigma\sqrt {h}$. 
\end{theorem}


\begin{theorem}\label{thmsiguna}  \cite[Theorem 3.1]{SST06}.
Suppose  
$0<\sigma\le 1$,  
$y\ge 1$,  
 $A\in \mathcal G_{0,\sigma}^{n\times n}$. Then the matrix $A$
 has full rank with 
probability $1$ and 
$F_{\kappa (A)}(y)\ge 1-(14.1+4.7\sqrt{(2\ln y)/n})n/ (y\sigma)$.
\end{theorem}


 
\begin{proof}
See \cite[the proof of Lemma 3.2]{SST06}.
\end{proof}


\section{Condition numbers of randomized matrix products and generic preconditioning}\label{smrc}


Next we deduce probabilistic lower bounds on the smallest 
singular values of the products of fixed and random matrices.
We begin with three lemmas. The first of them is obvious,
the second easily follows  
from minimax property (\ref{eqminmax}).


\begin{lemma}\label{lepr2}
$\sigma_{j}(SM)=\sigma_j(MT)=\sigma_j(M)$ for all $j$ if $S$ and $T$ are square orthogonal matrices.
\end{lemma}


\begin{lemma}\label{lepr1} 
Suppose $\Sigma=\diag(\sigma_i)_{i=1}^{n}$, $\sigma_1\ge \sigma_2\ge \cdots \ge \sigma_n$,
$G\in \mathbb R^{r\times n}$, $H\in \mathbb R^{n\times r}$.
Then 
$\sigma_{j}(G\Sigma)\ge\sigma_{j}(G)\sigma_n$,
$\sigma_{j}(\Sigma H)\ge\sigma_{j}(H)\sigma_n$ for all $j$.
If also $\sigma_n>0$, then 
 $\rank (G\Sigma)=\rank (G)$, $\rank (\Sigma H)=\rank (H)$.
\end{lemma}


\begin{lemma}\label{lepr3} \cite[Proposition 2.2]{SST06}.
Suppose $H\in \mathcal G_{\mu,\sigma}^{m\times n}$, $SS^T=S^TS=I_m$, $TT^T=T^TT=I_n$.
Then $SH\in \mathcal G_{\mu,\sigma}^{m\times n}$ and $HT\in \mathcal G_{\mu,\sigma}^{m\times n}$.
\end{lemma}

The following theorem implies that  
multiplication by
standard Gaussian random matrix is unlikely to decrease
the smallest positive singular value of a matrix
dramatically,  even though
$UV=O$ for some pairs of rectangular orthogonal matrices $U$ and $V$.

\begin{theorem}\label{1}
Suppose $G'\in \mathcal G_{\mu,\sigma}^{r\times m}$, $H'\in \mathcal G_{\mu,\sigma}^{n\times r}$,
 $M\in \mathbb R^{m\times n}$, 
$G=G'+U$, $H=H'+V$ for 
some matrices $U$ and $V$,
$r(M)=\rank (M)$, $x>0$ and
$y\ge 0$. 
Then
$F_{1/||(GM)^+||}(y)\le F(y,M,\sigma)$ and
$F_{1/||(MH)^+||}(y)\le F(y,M,\sigma)$
for $F(y,M,\sigma)=2.35 y \sqrt {\widehat r}||M^+||/\sigma$
and
$\widehat r=\min\{r,r(M)\}$,
that is 
$Probability \{||P^+||\ge 2.35x\sqrt {\widehat r}||M^+||/\sigma\}\le 1/x$
for $P=GM$ and $P=MH$.
\end{theorem} 



\begin{proof}
With probability $1$,
the matrix $MH$
has rank $\widehat r$ 
because $H\in \mathcal G_{\mu,\sigma}^{n\times r}$.
So (cf. (\ref{eqnrm+}))
\begin{equation}\label{eqfmw}
F_{1/||(MH)^+||}(y)=F_{\sigma_{\widehat r}(MH)}(y).
\end{equation}
 Let
$M=S_M\Sigma_MT^T_M$ be  full SVD where
$\Sigma_M=\diag(\widehat \Sigma_M, O)= \Sigma_M\diag(I_{r(M)},O)$ and
$\widehat \Sigma_M=\diag(\sigma_j(M))_{j=1}^{r(M)}$
is a nonsingular diagonal matrix. 
We have $MH=S_M\Sigma_MT_M^TH$, and so  
$\sigma_j(MH)=\sigma_j(\Sigma_MT_M^TH)$ for all $j$ 
by virtue of Lemma \ref{lepr2}, because 
$S_M$ is a square orthogonal matrix. 
Write 
$H_{r(M)}=(I_{r(M)}~|~O)T_M^TH$ and observe that
$\sigma_j(\Sigma_MT_M^TH)= \sigma_j(\widehat \Sigma_MH_{r(M)})$ 
 and consequently 
\begin{equation}\label{eqjmw}
\sigma_j(MH)= \sigma_j(\widehat \Sigma_MH_{r(M)})~{\rm for~all}~j.
\end{equation}
Combine equation
(\ref{eqjmw}) for $j=\widehat r$ with Lemma \ref{lepr1} 
for the pair $(\Sigma,H)$ replaced by $(\widehat \Sigma_M,H_{r(M)})$
 and 
obtain that 
$\sigma_{\widehat r}(MH)\ge \sigma_{r(M)}(M)\sigma_{\widehat r}(H_{r(M)})=
\sigma_{\widehat r}(H_{r(M)})/||M^+||$.
We have  $T_M^TH'\in \mathcal G_{\mu,\sigma}^{n\times r}$ 
by virtue of Lemma \ref{lepr3}, because $T_M$ is a square orthogonal matrix;
consequently  $H_{r(M)}=H_{r(M)}'+B$ for $H_{r(M)}' \in \mathcal G_{\mu,\sigma}^{r(M)\times r}$
and some matrix $B$.
Therefore we can apply
Theorem \ref{thsiguna} for $A=H'_{r(M)}$ 
and obtain the bound  of Theorem \ref{1} on $F_{1/||(MH)^+||}(y)$.
 One can similarly deduce the bound on $F_{1/||(GM)^+||}(y)$  
or can just 
apply the above bound on 
$F_{1/||(MH)^+||}(y))$ for $H=G^T$ and $M$ replaced by $M^T$
and then recall that $(M^TG^T)^T=GM$.
\end{proof}

By combining (\ref{eqnorm12inf}) with Theorems \ref{thsignorm} (for $B=O$) and \ref{1}
we can probabilistically bound the condition numbers of 
randomized  
products $GM$ and $MH$.
The following corollary extends the bound of Theorem \ref{1} 
for a randomized matrix product to the bounds for
its blocks. 


\begin{corollary}\label{cogh}
Suppose $j$, $k$, $m$, $n$, $q$ and $s$ are integers, $1\le j\le q$, $1\le k\le s$, 
$M\in \mathbb R^{m\times n}$, $\sigma>0$,
$G\in  \mathcal G_{\mu,\sigma}^{q\times m}$, $H\in \mathcal G_{\mu,\sigma}^{n\times s}$,
$\rank (M_j)=j$ for $M_j=M\begin{pmatrix}I_j   \\   O_{n-j,j}\end{pmatrix}$,
$\rank (M^{(k)})=k$ for $M^{(k)}=(I_k~|~O_{k,m-k})M$,    
and  $y\ge 0$.
Then 
(i) with probability $1$ the matrix $GM$ (resp. $MH$) has full rank 
if $\rank (M)\ge q$ (resp. if $\rank (M)\ge s$).
Furthermore
(ii) $F_{1/||((GM)_j^{(j)})^+||}(y)\le 2.35 y  \sqrt {j}||M_j^+||/\sigma$ 
 if $\rank (M)\ge j$,
$F_{1/||((MH)_k^{(k)})^+||}(y)\le 2.35 y  \sqrt {k}||(M^{(k)})^+||/\sigma$ if $\rank (M)\ge k$.
\end{corollary} 

\begin{proof}
We immediately verify part (i)
by applying the techniques of Section \ref{sngrm}. 
To prove part (ii) 
apply 
Theorem \ref{1}
replacing $G$ by $(I_j~|~O_{j,q-j})G$
and replacing $M$  by   
 $M\begin{pmatrix}I_j   \\   O_{n-j,j}\end{pmatrix}$. 
For every $k$ apply 
Theorem \ref{1}   
 replacing $M$ by $(I_k~|~O_{k,m-k})M$ and replacing $H$  by   
 $H\begin{pmatrix}I_k   \\   O_{s-k,k}\end{pmatrix}$. 
\end{proof}


Corollary \ref{cogh} can be immediately extended to any block 
of the matrices $GM$ and $MH$, 
but we single out the leading blocks
because 
 applications of
GENP 
and block Gaussian elimination
are numerically safe where 
these blocks are nonsingular 
and well conditioned.
We have empirical evidence 
that 
 such applications
are numerically safe 
even where we use
 circulant 
 multipliers $G$ 
and $H$ filled with $\pm 1$ and
where randomization is restricted to choosing
the signs $\pm$
(see Tables \ref{tab44} and \ref{tabSVD_HEAD1T}).
The paper \cite{T11} 
provides some  formal support for using some
other structured  randomized 
 multipliers.


\section{Approximate bases for singular spaces, 
low-rank 
approximation, and the
computation of numerical rank}\label{sapsr}


\subsection{Randomized low-rank approximation: an outline and an extension
to approximation by structured matrices}\label{sprb0}


Supppose we seek a rank-$\rho$ approximation to a matrix $ A$ 
that has a numerical rank $\rho$.
We can solve this problem by computing the SVD of the 
 matrix $A$
 or
 its rank-revealing 
 factorization \cite{GE96}, \cite{HP92}, \cite{P00a},
but in this section we cover 
alternative numerically stable and noncostly 
solutions
 based on
randomized matrix multiplication.  
As by-product we obtain 
approximate matrix bases   
for the left or right leading
singular space $\mathbb T_{\rho, A}$
and $\mathbb S_{\rho, A}$.


Let us supply further details. 
Our next theorem expresses
a rank-$\rho$ approximation to a matrix 
$ A$ through
a matrix basis 
for  any of 
the two leading
singular spaces $\mathbb T_{\rho, A}$
any $\mathbb S_{\rho, A}$.
 Theorem \ref{thover} of Section \ref{sprb2}
supports randomized computation of 
such an  approximate basis 
for the space $\mathbb T_{\rho, A}$
   from the
product $ A^TG$ for 
 $G\in \mathcal G_{0,1}^{m\times \rho_+}$.
The paper \cite{T11} formally supports this algorithm
for a special class of random structured multipliers
$G$, and
our tests 
consistently show such support 
where $G\in \mathcal T_{0,1}^{m\times \rho_+}$
(see Tables \ref{tabSVD_HEAD1} and \ref{tabSVD_HEAD1T}).
We conjecture that the same is  true for 
various other classes of sparse and structured multipliers
$G$, defined by much fewer random parameters
compared to the number of the entries.
We specify a low-rank approximation algorithm in Section \ref{sapsr0},
which has important applications 
to matrix computations, many listed in \cite{HMT11}.
For a natural extension
assume a  matrix $W$ 
having a possibly unknown numerical 
displacement rank $d$,
 that is lying near some matrices with a  
displacement rank $d$ (see the definitions in \cite{KKM79}, \cite{BM01}, \cite{p01}). 
We can compute one of these displacements as a rank-$d$
approximation to the displacement of the matrix $W$,
 and then 
immediately recover 
a structured matrix  approximating 
the matrix $W$.


\subsection{Low-rank approximation via the basis of 
a leading singular space}\label{sprb}


Next we prove
that both orthogonal and nonorthogonal projections 
of a matrix $A$ 
onto 
its leading singular spaces $\mathbb T_{\rho,A}$
and $\mathbb S_{\rho,A}$ approximate the matrix within 
the error norm $\sigma_{\rho+1}(A)$. 
 

\begin{theorem}\label{thsng} 
Suppose  $A$ is
an $m\times n$ matrix,
$S_A\Sigma_AT_A^T$ is its SVD of (\ref{eqsvd}), 
 $\rho$ is a positive integer,
$\rho\le l=\min\{m,n\}$, and
$T$ and $S$ are 
matrix 
bases for the spaces $\mathbb T_{\rho,A}$
and $\mathbb S_{\rho,A}$, respectively.
Then 
\begin{equation}\label{eqlrap}
||A-AT(T^TT)^{-1}T^T||=||A-S(S^TS)^{-1}S^TA||=\sigma_{\rho+1}(A).
\end{equation}
For orthogonal matrices $T$ and $S$ we have $T^TT=S^TS=I_{\rho}$
and  
\begin{equation}\label{eqlrap1}
||A-ATT^T||=||A-SS^TA||=\sigma_{\rho+1}(A).
\end{equation}
\end{theorem}
\begin{proof}
Write $P=T_{\rho,A}T_{\rho,A}^T$ and $r=n-\rho$,
observe that $T_A^TT_{\rho,A}=\begin{pmatrix}
I_{\rho}  \\
O_{r,\rho}
\end{pmatrix}$,
and obtain $AP=S_A\Sigma_AT_A^TT_{\rho,A}T_{\rho,A}^T=S_A\Sigma_A \begin{pmatrix}
T_{\rho,A}^T  \\
O_{r,\rho}
\end{pmatrix}$, whereas 
$A=S_A\Sigma_A \begin{pmatrix}
T_{\rho,A}^T  \\
T_{A,r}^T
\end{pmatrix}$. Hence
$A-AP=S_A\Sigma_A \begin{pmatrix}
O_{\rho,n}  \\
T_{A,r}^T
\end{pmatrix}=S_A\diag(O_{\rho,\rho},\Sigma_\rho) T_A^T$,
where $\Sigma_\rho$ is the $(m-\rho)\times (n-\rho)$ diagonal matrix
with the diagonal entries $\sigma_{\rho+1},\dots,\sigma_l$.
Thus
 $||A-AP||=||\Sigma_\rho||=\sigma_{\rho+1}(A)$
because $S_A$ and $T_{A}$ are square orthogonal matrices.
This proves  the estimates
(\ref{eqlrap}) 
and (\ref{eqlrap1})
for
$T=T_{\rho,A}$. Let us extend them to any matrix basis $T$ 
 for the space 
$\mathbb T_{\rho,A}$, that is
for $T=T_{\rho,A}U$ where $U$ is a 
nonsingular matrix.  
 Recall that $T_{\rho,A}^TT_{\rho,A}=I_{\rho}$, and obtain successively 
$(U^TT_{\rho,A}^TT_{\rho,A}U)^{-1}=U^{-1}U^{-T}$,
$U(U^TT_{\rho,A}^TT_{\rho,A}U)^{-1}U^T=I_{\rho}$,
and  $T(T^TT)^{-1}T^T=T_{\rho,A}U(U^TT_{\rho,A}^TT_{\rho,A}U)^{-1}U^TT_{\rho,A}^T=
T_{\rho,A}T_{\rho,A}^T$,
implying 
the desired extension.
Apply the proof to the transpose $A^T$ to
extend it to 
 matrix bases $S$ for the space $\mathbb S_{\rho,A}$.
\end{proof}


\subsection{A basis of 
a leading singular space via randomized products}\label{sprb2}


The following theorem supports randomized approximation of 
matrix bases for the leading singular spaces 
$\mathbb T_{\rho, A}$
and $\mathbb S_{\rho, A}$
of a matrix $ A$.
 

\begin{theorem}\label{thover} 
Suppose 
$ A\in \mathbb R^{m\times n}$,
$H\in \mathcal G_{0,1}^{n\times \rho}$, and
 $G\in \mathcal G_{0,1}^{m\times \rho}$
and write $S=AH$ and $T=A^TG$.
Then
(i) $\rank (T)= \rank (S)=\min \{\rho,\rank (A)\}$
 with probability $1$
and (ii) $\nrank (T)= \nrank (S)=\min \{\rho,\nrank (A)\}$
with probability close to $1$.
(iii) Furthermore 
with probability close to $1$
we  have 
\begin{equation}\label{eqlsap}
S+\Delta=S_{\rho, A}U~{\rm and}~ 
 T+\Delta'=T_{\rho, A}V
\end{equation} 
for two matrices $\Delta$ and $\Delta'$
having norms of order $\sigma_{\rho+1}( A)$
and for two nonsingular matrices $U$ and $V$
having condition numbers of at most order $||A||/(\sigma_{\rho}( A)\sqrt {\rho})$.  
\end{theorem}
\begin{proof}
We prove the claims about 
the matrix $S$, then  apply them to the transpose $A^T$
to extend 
 to the matrix $T$. 
The techniques of Section  \ref{sngrm} 
support part (i).
By 
truncating the SVD of the
matrix $A$ to the level of its numerical rank $\rho_-$ we 
obtain a well conditioned matrix having both rank and numerical rank  
$\rho_-$. Then we deduce part (ii)
from Theorem \ref{1}.

Proving part (iii)
we assume w.l.o.g.
that $\rank (A)\ge \rho$. 
(Otherwise 
infinitesimal perturbation of the matrix $A$
could yield this bound.)
Write SVD $A=
S_{ A}\Sigma_{ A}T_{ A}^T$,
 $\Sigma_{\rho, A}=\diag(\sigma_j(A))_{j=1}^{\rho}$,  
$U=\Sigma_{\rho, A}T_{\rho, A}^TH$,
and
$A_{\rho}=S_{\rho, A}U=S_{ A}~\diag(\Sigma_{\rho, A},O_{m-\rho,n-\rho})~T_{ A}^T=
S_{\rho, A}\Sigma_{\rho, A}T_{\rho, A}^T.$
Note that $||\Sigma_{ A}-\diag(\Sigma_{\rho, A},O_{m-\rho,n-\rho})||= \sigma_{\rho+1}(A)$. 
Consequently
$|| A-A_{\rho}||= \sigma_{\rho+1}(A)$,
$ AH=A_{\rho}H+\Delta$, where $||\Delta||\le \sigma_{\rho+1}(A)~||H||$,
and  the norm $||H||$  
is likely to be bounded from above and below by two positive constants 
(see Theorem \ref{thsignorm}). This implies  (\ref{eqlsap}).

With probability $1$ 
the $\rho\times \rho$ matrices $U=\Sigma_{\rho, A}B$ and $B=T_{\rho, A}H$ 
are nonsingular  
(see Section \ref{sngrm}). Next we assume that they are nonsingular
and estimate $\kappa(U)$.
Clearly $||U||\le ||\Sigma_{\rho, A}||~||T_{\rho, A}^T||~||H||$
where $||\Sigma_{\rho, A}||=||A||$, $||T_{\rho, A}^T||=1$.
Therefore $||U||\le ||A||~||H||$,
which is likely to have order $||A||$.

Furthermore 
we have 
$||U^{-1}||\le ||\Sigma_{\rho, A}^{-1}||~||B^{-1}||=
||B^{-1}||/\sigma_{\rho}(A)$.
Apply Theorem \ref{1} where $M=T_{\rho, A}^T$, 
$\widehat r=\rho$ and $\sigma_{r(M)}(M)=\sigma=1$
and obtain that the norm $||B^{-1}||$ is  likely to have at most order $1/\sqrt {\rho}$.
Therefore with probability close to $1$,
the norm $||U^{-1}||=||B^{-1||}/\sigma_{\rho}(A)$  has 
at most order $1/(\sigma_{\rho}(A)\sqrt {\rho})$, and then
 $\kappa (U)=||U||~||U^{-1}||$ has
 at most order $||A||/(\sigma_{\rho}( A)\sqrt {\rho})$.
\end{proof}


\subsection{A prototype algorithm for low-rank approximation}\label{sapsr0}


Together Theorems \ref{thsng} and \ref{thover} imply correctness of the following
prototype algorithm where
we assume that the input matrix has an unknown numerical rank
and we know its upper bound. 
The algorithm employs approximation of a leading singular space
of the input matrix.  


\begin{protoalgorithm}\label{algbasmp} {\bf Rank-$\rho$ approximation of a matrix}
 (cf. \cite[Section 10.3]{HMT11}).


\begin{description}


\item[{\sc Input:}] 
A  matrix 
$ A\in \mathbb R^{m\times n}$
having an unknown 
numerical rank $\rho$,
an integer $\rho_+\ge \rho$,
 and two tolerances $\tau$ and $\tau'$
of order $\sigma_{\rho+1}( A)/||A||$.
(We can  choose  $\tau$  at
Stage 2
based on  
rank revealing factorization
of an auxiliary $n\times \rho_+$
matrix. The computation of
this factorization 
is noncostly where $\rho$ is small.
We can  choose  $\tau'$  at
Stage 3 based 
on the required 
output accuracy,
and can  adjust both tolerances 
if  the algorithm fails
to produce a satisfactory output.)


\item[{\sc Output:}]
FAILURE (with a low probability) or 
an integer $\rho$ and two matrices 
$T\in \mathbb R^{n\times \rho}$ and 
$\widehat A_{\rho}\in \mathbb R^{m\times n}$, both having ranks 
at most $\rho$ and
such that $||\widehat A_{\rho}- A||\le \tau' ||A||$ and
$T$ satisfies bound (\ref{eqlsap}) of Theorem \ref{thover} 
for $||\Delta'||\le \tau ||A||$.


\item[{\sc Computations:}] $~$


\begin{enumerate}


\item 
 Compute the $n\times \rho_+$ matrix $T'= A^TG$ for  
$G\in \mathcal G_{0,1}^{m\times \rho_+}$.

\item 
Compute a rank revealing factorization of the matrix $T'$
and choose the minimal integer $s$ and an
$n\times s$ matrix  $T$ 
such that $||T'-(T~|~O_{n,\rho_+-s})||\le \tau ||A||$.

\item 
Compute the matrix
$\widehat A_s=AT(T^TT)^{-1}T^T$.
Output $\rho=s$, $T$ and $\widehat A_{\rho}$  and stop
if 
$||\widehat A_{\rho}- A||\le \tau' ||( A)||$.
Otherwise output
FAILURE  and stop.


\end{enumerate}


\end{description}


\end{protoalgorithm}
Assume a proper choice of
 both tolerances
$\tau$ and $\tau'$.
Then by virtue of Theorem \ref{thover},
we can expect that at Stage 2
we obtain
$s=\rho$  
and
an approximate matrix basis $T$ 
for the singular space  $ \mathbb T_{\rho, A}$
(within an error norm of at most order $\sigma_{\rho+1}( A)$).
If so, Stage 3  outputs FAILURE with a 
probability near $0$,
by virtue of Theorems \ref{thsng},
and in the case of FAILURE we can reapply the algorithm 
 for new values of random parameters
or for the adjusted tolerance values 
$\tau$ and $\tau'$.
At Stage 2 we have $s\le \rho$
because $\nrank(A^TG)\le \nrank(A)=\rho$,
whereas 
for a sufficiently small 
tolerance $\tau'$
the bound 
$||\widehat A_{\rho}- A||\le \tau' ||( A)||$
at Stage  3
implies that $s\ge \nrank(A)$.
These observations enable us to certify
correctness of the outputs $\rho$, $T$,
and $\widehat A_{\rho}$ of the algorithm.
 
We can similarly approximate the matrix $ A$ by
a rank-$\rho$ matrix $S(S^TS)^{-1}S^T A$, 
by first computing the matrix $S'=AH$ for 
$H\in \mathcal G_{0,1}^{n\times \rho_+}$,
then computing its rank revealing 
factorization, which is expected to define
an approximate matrix  
basis $S$ for the space $ \mathbb S_{\rho, A}$,
and finally applying Theorem \ref{thsng},
to approximate the matrix A by a rank-$\rho$ 
matrix.


\begin{remark}\label{respcl}
For $\rho_+=\rho$ we can write $T=T'=A^TG$
and skip Stage 2  because 
the matrix $ A^TG$ is expected to serve as a
desired 
approximate matrix basis
by virtue of Theorem  \ref{thover}.
In Appendix B we compute numerical rank by using
 randomized matrix multiplications
instead of standard recipes
that use orthogonalization or pivoting.
\end{remark}

\begin{remark}\label{reovers}
The increase of the dimension $\rho_+$
 beyond the numerical rank $\rho$
(called {\em oversampling} in  \cite{HMT11}) 
is relatively inexpensive
if the bound $\rho_+$
is small. \cite{HMT11} suggests 
using small
oversampling even if 
the numerical rank $\rho$ is known,
because
we have 
$${\rm Probability}~\{|| A- ATT^T||\le (1+9\sqrt{\rho_+\min\{m,n\}})\sigma_{\rho+1}(A)\}\ge
 1-3(\rho_+-\rho)^{\rho-\rho_+}~{\rm for}~\rho_+>\rho.$$
Theorem \ref{thover}, however,  bounds the norm
$||A- ATT^T||$ strongly also  
for $\rho=\rho_+$, 
 in good accordance with the data
of 
Tables \ref{tabSVD_HEAD1} and
\ref{tabSVD_HEAD1T}. 
\end{remark}

\begin{remark}\label{restrm}
For a larger integer $\rho$
we can substantially simplify Stage 1 of
 the algorithm
by choosing 
structured multipliers $G$ from the class of
 the {\em subsample random Fourier transforms}, 
called {\em SRFTs}. Under this choice 
the estimated probability of 
obtaining low rank approximation 
is close to the above case of Gaussian random
multipliers $G$ \cite{T11}.
Our tests in Section \ref{sexp}
provide informal empirical 
support for  similar use 
of random Toeplitz multipliers $G$.
\end{remark}

\begin{remark}\label{reort} 
By applying rank revealing QR factorization at Stage 2
of the algorithm we can produce an orthogonal 
matrix $T$ and consequently simplify Stage 3
by computing $\widehat A_s=ATT^T$ 
(cf. (\ref{eqlrap1})). 
We adopted such a variation of the algorithm
in our tests in Section \ref{sexp}. 
\end{remark}

\begin{remark}\label{resmpl}
One can weaken
 reliability of the output
to
simplify Stage 3 by testing 
whether 
$||K^T( A-\widehat A_{\rho})L||\le \tau ||K||~|| A||~||L||$
for matrices $K\in \mathcal G_{0,1}^{m\times \rho'}$
and $L\in \mathcal G_{0,1}^{n\times \rho''}$ and
for two small positive integers $\rho'$ and $\rho''$,
possibly for $\rho'=\rho''=1$,
instead of testing 
whether 
$||\widehat A_{\rho}- A||\le \tau' ||( A)||$.  
One can 
similarly simplify Stage 2.
\end{remark}


\begin{remark}\label{regap}
Application of Proto-Algorithm \ref{algbasmp}
to the approximation of the leading singular spaces 
$\mathbb T_{\rho, A}$
and $\mathbb S_{\rho, A}$
is facilitated and its power is enhanced 
as the gaps increase between the singular
values of the input matrix $A$.
This motivates using
the power transforms $A\Longrightarrow B_h=(AA^T)^hA$
 for positive integers $h$
because $\sigma_j(B_h)=(\sigma_j(A))^{2h+1}$ for all $j$.
\end{remark}


\begin{remark}\label{reldtr}
Clearly every matrix basis 
of the trailing singular space $\mathbb T_{A, n-\rho}$
of an $m\times n$ matrix $A$
is orthogonal to
every matrix basis 
of the leading singular space $\mathbb T_{\rho,A}$,
 and similarly for any pair 
of matrix bases of
the spaces $\mathbb S_{A, m-\rho}$
and $\mathbb S_{\rho,a}$.
One can exploit this duality for the computation
and approximation of the bases.
\end{remark}


\section{Application to Tensor Train decomposition}\label{svianvtns}


Let 
\begin{equation}\label{eqtens}
{\bf A}=[A(i_1,\dots,i_d)]
\end{equation}
denote a $d$-dimensional {\em tensor} with entries 
$A(i_1,\dots,i_d)$ and {\em spacial indices}
$i_1,\dots,i_d$ ranging from $1$ to 
$n_1,\dots,n_d$, respectively.
Define the $d-1$ {\em unfolding matrices}
$A_k=[A(i_1\dots i_k;i_{k+1}\dots i_d)],~k=1,\dots,d$,
where the semicolon separates the
multi-indices
$i_1\dots i_k$ and $i_{k+1}\dots i_d$,
which define the rows and columns of the matrix $A_k$,
respectively, $k=1,\dots,d$.
The paper \cite{O09} proposed the following class of {\em Tensor Train Decompositions},
 hereafter referred to as {\em TT Decompositions}, where
the {\em summation indices} $\alpha_{1},\dots,\alpha_{d-1}$
ranged from $1$ to {\em compression ranks} $r_1,\dots,r_{d-1}$,
respectively,
\begin{equation}\label{eqtt}
T=\sum_{\alpha_1,\dots,\alpha_{d-1}}G_1(i_1,\alpha_1)G_2(\alpha_1,i_1,\alpha_2)\cdots
G_{d-1}(\alpha_{d-2},i_{d-1},\alpha_{d-1})G_d(\alpha_d,i_d).
\end{equation}


\begin{theorem}\label{thtt} \cite{O09}.
For any tensor ${\bf A}$ of (\ref{eqtens})
there exists a TT decomposition (\ref{eqtt})
such that ${\bf A}={\bf T}$ and $r_k=\rank (A_k)$
for $k=1,\dots,d-1$.
\end{theorem}
There is a large and growing number
of important applications of TT decompositions 
(\ref{eqtt})
to modern computations
(cf. e.g., \cite{OT09}, \cite{OT10}, \cite{OT11})
where the numerical ranks 
of the unfolding matrices $A_k$
are much smaller than their ranks, and
 it is desired to
 compress TT decompositions respectively.


\begin{theorem}\label{thot} \cite{OT10}.
For any tensor ${\bf A}$ of (\ref{eqtens})
and any set of positive integers 
 $r_k\le \rank (A_k)$,  $k=1,\dots,d-1$,
there exists a TT decomposition (\ref{eqtt})
such that
\begin{equation}\label{eqot}
||{\bf A}-{\bf T}||_F^2\le \sum_{k=1}^{d-1}\tau_k^2,~\tau_k=\min_{\rank (B)=r_k}||A_k-B||_F,~k=1,\dots,d-1. 
\end{equation}
\end{theorem}

The constructive proof of this theorem in \cite{OT10}
relies on inductive approximation of unfolding matrices 
by their SVDs truncated to the compression ranks $r_k$.
Let us sketch this  construction.
For $d=2$ we obtain  a desired TT
decomposition
$T(i_1,i_2)=\sum_{\alpha_1}^{r_1}G_1(i_1,\alpha_1)G_2(\alpha_1,i_2)$
 (that is a sum of $r_1$
outer products of $r_1$ pairs of vectors) simply
 by  truncating 
 the SVD of the matrix $A(i_1,i_2)$.
At the inductive step one truncates the SVD 
of the first unfolding matrix 
$A_1=S_{A_1}\Sigma_{A_1}T_{A_1}^T$
to obtain rank-$r_1$ approximation of this matrix
$B_1=S_{B_1}\Sigma_{B_1}T_{B_1}^T$
 where $\Sigma_{B_1}=\diag(\sigma_j(A_1))_{j=1}^{r_1}$
and the matrices $S_{B_1}$ and $T_{B_1}$
are formed by the first $r_1$ columns of the matrices 
$S_{A_1}$ and $T_{A_1}$, respectively.
Then it remains to approximate the tensor
${\bf B}=[B(i_1,\dots,i_d)]$
represented by the matrix $B_1$.
Rewrite it as
$\sum_{\alpha_1=1}^nS_{B_1}(i_1;\alpha_1)\widehat A(\alpha_1;i_2\dots i_d)$
for $\widehat A=\sum_{B_1}T_{B_1}^T$,
represent $\widehat A$ as the tensor
${\bf \widehat A}=[A(\alpha_1i_2,i_3,\dots,i_d)]$
of dimension $d-1$, apply the inductive hypothesis
to obtain a TT-approximation of this tensor,
and extend it to a TT-approximation of the original tensor 
${\bf A}$.

In \cite{OT10} the authors specify this construction as 
their Algorithm 1,
 prove error norm bound (\ref{eqot}), then
point out that the 
``computation of the truncated SVD 
for large scale and possibly dense 
unfolding matrices ... is unaffordable
in many dimensions", propose 
``to replace SVD by some other dyadic decompositions
$A_k\approx UV^T$,
which can be computed with low complexity",  and 
finally specify such recipe as  \cite[Algorithm 2]{OT10},
 which
is an iterative algorithm for 
skeleton or pseudoskeleton decomposition
of matrices and which they
 use at Stages 5 and 6
of their Algorithm 1.
The cost of each iteration 
of \cite[Algorithm 2]{OT10}
is quite low, 
and empirically the iteration converges fast, but
the authors welcome alternative recipes
 having formal support.

Proto-Algorithm
\ref{algbasmp} 
can
serve as an alternative to \cite[Algorithm 2]{OT10}.
For the input matrix $A_1$ above
we  
use $O(r_1)$ multiplications of this matrix
by $O(r_1)$ vectors, which means a low 
computational cost
 for sparse and structured inputs,
whereas the expected output is 
an approximate matrix basis for the space $\mathbb S_{r_1,A_1}$
or  $\mathbb T_{r_1,A_1}$ and 
a  rank-$r_1$ approximation to the matrix $A_1$,
within an expected error norm in
$O(\sigma_{r_1+1}(A_1))$. This is  the same order as
in  \cite[Algorithm 1]{OT10}, but now we do not use SVDs.
 One can further decrease the error bound by means
of small oversampling of Remark \ref{reovers} 
and the power transform of Remark
\ref{regap}. 
     
\begin{remark}\label{retns}
A huge bibliography on tensor decompositions 
and on thier application to fundamental matrix computations
has been recently surveyed in \cite{KB09}, but
with the omission of
 the early works
\cite{P72}, \cite{P79}, 
 \cite{B80},  \cite{P84}, 
 \cite{B86},
where nontrivial tensor decompositions helped to
accelerate the fundamental operation of matrix multiplication,
probably the first application of this kind.
\end{remark} 


\section{Numerical Experiments}\label{sexp}

 
Our numerical experiments with random general, Hankel, Toeplitz and circulant matrices 
have been performed in the Graduate Center of the City University of New York 
on a Dell server with a dual core 1.86 GHz
Xeon processor and 2G memory running Windows Server 2003 R2. The test
Fortran code has been compiled with the GNU gfortran compiler within the Cygwin
environment.  Random numbers have been generated with the random\_number
intrinsic Fortran function, assuming the uniform probability distribution 
over the range $\{x:-1 \leq x < 1\}$.  The tests have been designed by the first author 
and performed by his coauthor.


\subsection{GENP
with random circulant multipliers}\label{sexgeneral}


Table \ref{tab44} shows the results of our tests of the solution 
of a nonsingular well conditioned linear system $A{\bf y}={\bf b}$ of $n$ equations 
whose coefficient matrix has ill conditioned
$n/2\times n/2$   leading principal block for $n=64, 256,1024$.
We have performed 100 numerical tests for each dimension $n$ and computed 
the maximum, minimum and average relative residual norms 
$||A{\bf y}-{\bf b}||/||{\bf b}||$
as well as standard deviation. 
GENP applied to these systems outputs corrupted solutions with residual norms
ranging from 10 to $10^8$. When we preprocessed the systems with circulant multipliers
filled with $\pm 1$ (choosing the $n$ signs $\pm$ at random), the norms decreased to at
worst $10^{-7}$ for all inputs. Table \ref{tab44} also shows further decrease of the norm
in a single step of iterative refinement. Table 2 in \cite{PQZa}
shows similar results of the tests where  
the input matrices have been chosen similarly
but so that their every leading $k\times k$ block 
 had numerical rank $k$ or $k-1$
and where  Householder multipliers $I_n-{\bf u}{\bf v}^T/{\bf u}^T{\bf v}$
replaced the circulant multipliers. Here ${\bf u}$ and ${\bf v}$
denote two vectors filled with integers $1$ and $-1$ under 
random choice of the signs $+$ and $-$.


\subsection{Approximation of the tails and heads of SVDs and 
low-rank appro\-xi\-ma\-tion of 
a matrix }\label{stails}  


At some specified stages of our tests of this subsection 
we performed  
additions, subtractions and multiplications
with infinite precision
(hereafter referred to as {\em error-free ring operations}). 
At the other stages we performed computations 
 with double 
precision, and we rounded to 
double 
precision
all random values. We performed at most two refinement iterations for 
the computed solution of every linear system of equations
and matrix inverse.

Tables \ref{tabSVD_HEAD1} and \ref{tabSVD_HEAD1T}  
display the data from our tests on the approximation
of leading singular spaces of the SVD of an $n\times n$
matrix $A$ having 
numerical rank $q$ and on the
 approximation of this matrix with a matrix of rank $\rho$. 
For $n=64, 128, 256$ and  $\rho=1,8,32$ we  generated $n\times n$ random 
orthogonal matrices $S$ and $T$ and diagonal matrices 
$\Sigma=\diag(\sigma_j)_{j=1}^n$ such that $\sigma_j=1/j,~j=1,\dots,\rho$,
$\sigma_j=10^{-10},~j=\rho+1,\dots,n$ (cf. \cite[Section 28.3]{H02}). 
Then we applied error-free  ring operations to compute the input matrices
$A=S_A\Sigma_A T_A^T$, for which $||A||=1$ and $\kappa (A)=10^{10}$.
Furthermore 
we  generated 
 random $n\times \rho$ matrices $G$ (for $\rho=1,8,32$)
and successively computed the matrices
$B_{\rho,A}=A^TG$,
 $T_{\rho,A}$,  
$B_{\rho,A}Y_{\rho,A}$  as a least-squares approximation to $T_{\rho,A}$, 
$Q_{\rho,A}=Q(B_{\rho,A})$ (cf. Fact \ref{faqrf}), and 
$A-AQ_{\rho,A}(Q_{\rho,A})^T$ (by applying error-free  ring operations). 
Table \ref{tabSVD_HEAD1}  summarizes the data on  the residual norms 
${\rm rn}^{(1)}=||B_{\rho,A}Y_{\rho,A}-T_{\rho,A}||$ 
and
${\rm rn}^{(2)}=||A-AQ_{\rho,A}(Q_{\rho,A})^T||$ 
obtained in 100 runs of our tests
for every pair of $n$ and $\rho$.

We have also performed similar tests where  
we generated random Toeplitz $n\times \rho$ matrices $T$ 
(for $\rho=8,32$)
and then
replaced the above approximate matrix bases 
 $B_{\rho,A}=A^TG$
for the leading singular space $\mathbb T_{\rho,A}$
by the matrices $B_{\rho,A}=A^TT$.
Table \ref{tabSVD_HEAD1T}  displays the results of these tests.
In both Tables \ref{tabSVD_HEAD1} and \ref{tabSVD_HEAD1T} the residual 
norms are more or less equally small.






\begin{table}[ht]
  \caption{Relative residual norms:
 randomized circulant GENP for  
well conditioned linear systems with
ill conditioned leading blocks
  (cf. \cite[Table 2]{PQZa})}
  \label{tab44}
  \begin{center}
    \begin{tabular}{| c | c | c | c | c | c | c |}
      \hline
      \bf{dimension} & \bf{iterations} & \bf{min} & \bf{max} & \bf{mean} & \bf{std} \\ \hline
 $64$ & $0$ & $4.7\times 10^{-14}$ & $8.0\times 10^{-11}$ & $4.0\times 10^{-12}$ & $1.1\times 10^{-11}$ \\ \hline
 $64$ & $1$ & $1.9\times 10^{-15}$ & $5.3\times 10^{-13}$ & $2.3\times 10^{-14}$ & $5.4\times 10^{-14}$ \\ \hline
 $256$ & $0$ & $1.7\times 10^{-12}$ & $1.4\times 10^{-7}$ & $2.0\times 10^{-9}$ & $1.5\times 10^{-8}$ \\ \hline
 $256$ & $1$ & $8.3\times 10^{-15}$ & $4.3\times 10^{-10}$ & $4.5\times 10^{-12}$ & $4.3\times 10^{-11}$ \\ \hline
 $1024$ & $0$ & $1.7\times 10^{-10}$ & $4.4\times 10^{-9}$ & $1.4\times 10^{-9}$ & $2.1\times 10^{-9}$ \\ \hline
 $1024$ & $1$ & $3.4\times 10^{-14}$ & $9.9\times 10^{-14}$ & $6.8\times 10^{-14}$ & $2.7\times 10^{-14}$ \\ \hline
    \end{tabular}
  \end{center}
\end{table}



\begin{table}[h] 
  \caption{Heads of SVDs and 
low-rank approximation by using random multipliers $G$}
\label{tabSVD_HEAD1}
  \begin{center}
    \begin{tabular}{| c | c | c | c | c | c |c|}
      \hline
$q$  & ${\rm rrn}_i$ & n & \bf{min} & \bf{max} & \bf{mean} & \bf{std}\\ \hline
1 & ${\rm rn}^{(1)}$ & 64 &  $ 2.35\times 10^{-10} $  &  $ 1.32\times 10^{-07} $  &  $ 3.58\times 10^{-09} $  &  $ 1.37\times 10^{-08} $  \\ \hline		
1 & ${\rm rn}^{(1)}$ & 128 &  $ 4.41\times 10^{-10} $  &  $ 3.28\times 10^{-08} $  &  $ 3.55\times 10^{-09} $  &  $ 5.71\times 10^{-09} $  \\ \hline		
1 & ${\rm rn}^{(1)}$ & 256 &  $ 6.98\times 10^{-10} $  &  $ 5.57\times 10^{-08} $  &  $ 5.47\times 10^{-09} $  &  $ 8.63\times 10^{-09} $  \\ \hline		
1 & ${\rm rn}^{(2)}$ & 64 &  $ 8.28\times 10^{-10} $  &  $ 1.32\times 10^{-07} $  &  $ 3.86\times 10^{-09} $  &  $ 1.36\times 10^{-08} $  \\ \hline		
1 & ${\rm rn}^{(2)}$ & 128 &  $ 1.21\times 10^{-09} $  &  $ 3.28\times 10^{-08} $  &  $ 3.91\times 10^{-09} $  &  $ 5.57\times 10^{-09} $  \\ \hline		
1 & ${\rm rn}^{(2)}$ & 256 &  $ 1.74\times 10^{-09} $  &  $ 5.58\times 10^{-08} $  &  $ 5.96\times 10^{-09} $  &  $ 8.47\times 10^{-09} $  \\ \hline		
8 & ${\rm rn}^{(1)}$ & 128 &  $ 2.56\times 10^{-09} $  &  $ 1.16\times 10^{-06} $  &  $ 4.30\times 10^{-08} $  &  $ 1.45\times 10^{-07} $  \\ \hline		
8 & ${\rm rn}^{(1)}$ & 256 &  $ 4.45\times 10^{-09} $  &  $ 3.32\times 10^{-07} $  &  $ 3.40\times 10^{-08} $  &  $ 5.11\times 10^{-08} $  \\ \hline		
8 & ${\rm rn}^{(2)}$ & 64 &  $ 1.46\times 10^{-09} $  &  $ 9.56\times 10^{-08} $  &  $ 5.77\times 10^{-09} $  &  $ 1.06\times 10^{-08} $  \\ \hline		
8 & ${\rm rn}^{(2)}$ & 128 &  $ 1.64\times 10^{-09} $  &  $ 4.32\times 10^{-07} $  &  $ 1.86\times 10^{-08} $  &  $ 5.97\times 10^{-08} $  \\ \hline		
8 & ${\rm rn}^{(2)}$ & 256 &  $ 2.50\times 10^{-09} $  &  $ 1.56\times 10^{-07} $  &  $ 1.59\times 10^{-08} $  &  $ 2.47\times 10^{-08} $  \\ \hline		
32 & ${\rm rn}^{(1)}$ & 64 &  $ 6.80\times 10^{-09} $  &  $ 2.83\times 10^{-06} $  &  $ 1.01\times 10^{-07} $  &  $ 3.73\times 10^{-07} $  \\ \hline		
32 & ${\rm rn}^{(1)}$ & 128 &  $ 1.25\times 10^{-08} $  &  $ 6.77\times 10^{-06} $  &  $ 1.28\times 10^{-07} $  &  $ 6.76\times 10^{-07} $  \\ \hline		
32 & ${\rm rn}^{(1)}$ & 256 &  $ 1.85\times 10^{-08} $  &  $ 1.12\times 10^{-06} $  &  $ 1.02\times 10^{-07} $  &  $ 1.54\times 10^{-07} $  \\ \hline		
32 & ${\rm rn}^{(2)}$ & 64 &  $ 1.84\times 10^{-09} $  &  $ 6.50\times 10^{-07} $  &  $ 2.30\times 10^{-08} $  &  $ 8.28\times 10^{-08} $  \\ \hline		
32 & ${\rm rn}^{(2)}$ & 128 &  $ 3.11\times 10^{-09} $  &  $ 1.45\times 10^{-06} $  &  $ 2.87\times 10^{-08} $  &  $ 1.45\times 10^{-07} $  \\ \hline		
32 & ${\rm rn}^{(2)}$ & 256 &  $ 4.39\times 10^{-09} $  &  $ 2.16\times 10^{-07} $  &  $ 2.37\times 10^{-08} $  &  $ 3.34\times 10^{-08} $  \\ \hline		

    \end{tabular}
  \end{center}
\end{table}


\begin{table}[h] 
  \caption{Heads of SVDs and 
low-rank approximations by using  random Toeplitz multipliers $T$}
\label{tabSVD_HEAD1T}
  \begin{center}
    \begin{tabular}{| c | c | c | c | c | c |c|}
      \hline
$q$  & ${\rm rrn}^{(i)}$ & n & \bf{min} & \bf{max} & \bf{mean} & \bf{std}\\ \hline

8 & ${\rm rrn}^{(1)}$ & 64 &  $ 2.22\times 10^{-09} $  &  $ 7.89\times 10^{-06} $  &  $ 1.43\times 10^{-07} $  &  $ 9.17\times 10^{-07} $  \\ \hline		
8 & ${\rm rrn}^{(1)}$ & 128 &  $ 3.79\times 10^{-09} $  &  $ 4.39\times 10^{-05} $  &  $ 4.87\times 10^{-07} $  &  $ 4.39\times 10^{-06} $  \\ \hline				
8 & ${\rm rrn}^{(1)}$ & 256 &  $ 5.33\times 10^{-09} $  &  $ 3.06\times 10^{-06} $  &  $ 6.65\times 10^{-08} $  &  $ 3.12\times 10^{-07} $  \\ \hline			
8 & ${\rm rrn}^{(2)}$ & 64 &  $ 1.13\times 10^{-09} $  &  $ 3.66\times 10^{-06} $  &  $ 6.37\times 10^{-08} $  &  $ 4.11\times 10^{-07} $  \\ \hline				
8 & ${\rm rrn}^{(2)}$ & 128 &  $ 1.81\times 10^{-09} $  &  $ 1.67\times 10^{-05} $  &  $ 1.90\times 10^{-07} $  &  $ 1.67\times 10^{-06} $  \\ \hline			
8 & ${\rm rrn}^{(2)}$ & 256 &  $ 2.96\times 10^{-09} $  &  $ 1.25\times 10^{-06} $  &  $ 2.92\times 10^{-08} $  &  $ 1.28\times 10^{-07} $  \\ \hline	
32 & ${\rm rrn}^{(1)}$ & 64 &  $ 6.22\times 10^{-09} $  &  $ 5.00\times 10^{-07} $  &  $ 4.06\times 10^{-08} $  &  $ 6.04\times 10^{-08} $  \\ \hline
32 & ${\rm rrn}^{(1)}$ & 128 &  $ 2.73\times 10^{-08} $  &  $ 4.88\times 10^{-06} $  &  $ 2.57\times 10^{-07} $  &  $ 8.16\times 10^{-07} $  \\ \hline
32 & ${\rm rrn}^{(1)}$ & 256 &  $ 1.78\times 10^{-08} $  &  $ 1.25\times 10^{-06} $  &  $ 1.18\times 10^{-07} $  &  $ 2.03\times 10^{-07} $  \\ \hline	
32 & ${\rm rrn}^{(2)}$ & 64 &  $ 1.64\times 10^{-09} $  &  $ 1.26\times 10^{-07} $  &  $ 9.66\times 10^{-09} $  &  $ 1.48\times 10^{-08} $  \\ \hline
32 & ${\rm rrn}^{(2)}$ & 128 &  $ 5.71\times 10^{-09} $  &  $ 9.90\times 10^{-07} $  &  $ 5.50\times 10^{-08} $  &  $ 1.68\times 10^{-07} $  \\ \hline	
32 & ${\rm rrn}^{(2)}$ & 256 &  $ 4.02\times 10^{-09} $  &  $ 2.85\times 10^{-07} $  &  $ 2.74\times 10^{-08} $  &  $ 4.48\times 10^{-08} $  \\ \hline	
    \end{tabular}
  \end{center}
\end{table}


\section{Conclusions}\label{sconcl}


It is well known that random matrices 
tend to be well conditioned,
and  
this property motivates our application of 
random matrix multipliers
for advancing some fundamental
 matrix 
computations.
We first prove the basic fact 
that
with 
 a probability close to $1$ 
multiplication  
by a Gaussian random 
matrix does not increase 
the condition number of a matrix
and of its any block
dramatically
compared to the condition
number of the input matrix.
As an immediate implication 
random  multipliers are
likely to stabilize 
numerically 
 GENP
(that is Gaussian elimination with no pivoting)
and block Gaussian elimination
applied to a nonsingular and well conditioned matrix,
possibly having ill conditioned and singular 
leading blocks.
by applying to input matrix randomized
structured multipliers. 
Another basic fact states 
 that 
with 
 a probability close to $1$
the column sets of
the products $A^TG$ and $AH$
where an $m\times n$ matrix $A$ has a numerical rank $\rho$
and $G$ and $H$ are Gaussian random matrices
of sizes $m\times \rho$ and $n\times \rho$,
respectively, approximate some bases for the left and right 
leading singular
spaces $ \mathbb S_{\rho, A}$
and $ \mathbb T_{\rho, A}$
associated with the $\rho$ largest singular values
of the matrix $A$.  
Having any of such approximate bases 
available we can readily approximate 
the matrix $A$ by a matrix of rank $\rho$,
This has  further well known extensions to many
important matrix computations, 
and we point out  new extensions
to the
approximation of a matrix by a structured 
matrix lying  nearby,
to computing numercial rank of a matrix, 
and to
 Tensor Train approximation.  

Our tests consistently showed efficiency
of the proposed techniques even where
instead of general Gaussian random multipliers
we applied structured and sparse multipliers.
In these cases
randomization was limited to much fewer
random parameters or just
to the choice of the signs $\pm$ 
of a few auxiliary vectors. 
The recent paper \cite{T11} 
is an important step
toward
understanding and exploiting 
this phenomenon and should motivate further research
effort. 
Another natural research subject
 is the combination
of randomized matrix 
multiplication with
randomized techniques of
additive preprocessing and 
 augmentation,  
recently studied in  \cite{PGMQ}, \cite{PIMR10}, 
\cite{PQ10}, \cite{PQ12},
\cite{PQZC}, \cite{PQZa}, \cite{PQZb},  and \cite{PQZc}.

\bigskip




{\bf {\LARGE {Appendix}}}
\appendix 


\section{Uniform random sampling and nonsingularity of random matrices}\label{srsnrm}


{\em Uniform random sampling} of elements from a finite set $\Delta$ is their selection   
from  this set at random, independently of each other and
under the uniform probability distribution on the set $\Delta$. 


\begin{theorem}\label{thdl} 
Under the assumptions of Lemma \ref{ledl} let the values of the variables 
of the polynomial be randomly and uniformly sampled from a finite set $\Delta$. 
Then the polynomial vanishes with a probability at most $\frac{d}{|\Delta|}$. 
\end{theorem}


\begin{corollary}\label{codlstr} 
Let the entries of a general or Toeplitz  $m\times n$ 
matrix have been randomly and uniformly 
sampled from a finite set $\Delta$ of cardinality $|\Delta|$ (in any fixed ring). 
Let $l=\min\{m,n\}$.
Then (a) every $k\times k$ submatrix $M$ for $k\le l$ is nonsingular with a probability at 
least $1-\frac{k}{|\Delta|}$ and (b) is strongly nonsingular with a probability at least 
$1-\sum_{i=1}^k\frac{i}{|\Delta|}= 1-\frac{(k+1)k}{2|\Delta|}$.
\end{corollary}


\begin{proof}
The claimed properties of nonsingularity and nonvanishing hold for generic matrices. 
The singularity of a $k\times k$ matrix means that its determinant vanishes,
but the determinant is a polynomial of total degree $k$ in the entries. Therefore
Theorem \ref{thdl} implies
parts (a) and consequently (b). Part (c) follows because a fixed entry of the inverse vanishes
if and only if the respective entry of the adjoint vanishes, but up to the sign the latter 
entry is the determinant of a $(k-1)\times (k-1)$ submatrix of the input matrix $M$, and so it is
a polynomial of degree $k-1$ in its entries. 
\end{proof}


\section{Computation of numerical ranks}\label{stlss1}


The customary algorithms for the 
numerical  rank of a matrix 
rely on computing its SVD or 
rank revealing factorization
and involving 
pivoting or orthogonalization.
 Proto-Algorithm 
\ref{algbasmp} also
uses
rank revealing factorization
at Stage 2 and matrix inversion or 
orthogonalization at Stage 3,
but only 
with matrices of small 
sizes
provided the integer $\rho_+$ is small.
Our next 
alternative  
 algorithm
avoids pivoting and
orthogonalization 
even where
the numerical  rank $\rho$ is
large.
As by-product we compute an
approximate matrix basis 
within an error norm in $O(\sigma_{\rho+1}(\tilde A))$
for 
the leading  singular space 
$ \mathbb T_{\rho, \tilde A}$ 
of an $m\times n$ matrix $ \tilde A$ 
and 
can extend this readily to 
computing a
rank-$\rho$ 
approximation of the matrix $\tilde A$
(see Remark \ref{recorr}).
 We let
$m\ge n$
(else shift to $ \tilde A^T$),
let
$[\rho_-,\rho_+]=[0,n]$
unless we know a more narrow range,
and successively test the selected 
candidate integers in the  range
$[\rho_-,\rho_+]$
until we find the numerical rank $\rho$. 
To improve reliability, we can 
repeat the tests for
distinct values of random parameters
(see Remarks \ref{recorr}).

Exhaustive search
 defines and verifies the numerical rank $\rho$
with probability near $1$, but with proper
policies one can use fewer and simpler tests
because  
for $G\in \mathcal G_{0,1}^{m\times s}$
(and empirically for various random sparse and 
structured matrices $G$ as well)
 the matrix $B= \tilde A^TG$ is expected 
(a) to have full rank and  to be
well conditioned 
if and only if $s\ge \rho$,
(b) to  
 approximate a
matrix basis    
(within an error norm 
in $O(\sigma_{\rho+1}(\tilde A))$)
for a linear space   
$\mathbb T\supseteq \mathbb T_{\rho,B}=\mathbb T_{\rho, \tilde A}$ 
where
$s\ge \rho$, and (c)  to
approximate 
a
matrix basis
(within an error norm 
in $O(\sigma_{\rho+1}(\tilde A))$)
 for the space 
$\mathbb T_{\rho, \tilde A}$ where $s=\rho$.
Property (a) is implied by Theorem \ref{1},
properties (b) and (c) by Theorem \ref{thover}.


\begin{protoalgorithm}\label{algnrank} {\bf Numerical rank
without pivoting and orthogonalization} 
(see Remarks \ref{recorr}--\ref{rebsch}).


\begin{description}


\item[{\sc Input:}] 
Two integers $\rho_-$ and $\rho_+$ and a
matrix 
$ \tilde A\in \mathbb R^{m\times n}$ 
having unknown numerical rank $\rho=\rank ( \tilde A)$
in the range $[\rho_-,\rho_+]$
such that $0\le \rho_-< \rho_+\le n\le m$,
a rule for the selection of a candidate 
integer
$\rho$ in a range $[\rho_-,\rho_+]$,
and  a Subroutine COND
that determines whether a given matrix 
has  full rank and is well 
conditioned or not.


\item[{\sc Output:}] 
an integer  $\rho$ expected to equal
numerical rank 
of the matrix 
$\tilde A$ and
a
matrix $B$
 expected to
approximate 
(within an error norm 
in $O(\sigma_{\rho+1}(\tilde A))$)
a matrix basis
of the 
singular space  $\mathbb T_{\rho, \tilde A}$.
(Both expectations can actually fail, but
with a low probability, see Remark \ref{recorr}.)


\item[{\sc Initialization:}] 
Generate matrix $G\in \mathcal G_{0,1}^{m\times \rho_+}$ and
write  $B= \tilde A$, $G_{\rho}=G(I_{\rho}~|~O_{\rho,m-\rho})^T$ for $\rho=\rho_-,\rho_-+1,\dots,\rho_+$
(The $m\times \rho$ matrix $G_{\rho}$ is formed by the first $\rho$ columns of the matrix $G$.)


\item{\sc Computations}: $~$ 


\begin{enumerate}


\item 
Stop and output $\rho=\rho_+$ and the matrix $B$
if $\rho_-=\rho_+$.
Otherwise fix an integer $\rho$ in the range $[\rho_-,\rho_+]$.


\item 
Compute the matrix $B'=B^TG_{\rho}$
and apply to it the Subroutine COND. 


\item 
If this
matrix  has full rank
and is well 
conditioned,
write 
$\rho_+=\rho$
and $B= B'$
and go to Stage 1.
Otherwise 
write $\rho_-=\rho$
and go to Stage 1.


\end{enumerate}


\end{description}


\end{protoalgorithm} 

\begin{remark}\label{recorr}
The algorithm can output a wrong value 
of the numerical rank, although 
by virtue of Theorems \ref{thsng} and \ref{thover}
combined
this occurs with
a 
low probability. 
One can decrease this probability
by reapplying
the algorithm
to the same inputs and 
choosing distinct random 
parameters. Furthermore one can
fix
a 
tolerance $\tau$
of order $\sigma_{\rho+1}(\tilde A)$,
set $T=B$, and apply Stage 3 of 
Proto-Algorithm \ref{algbasmp}.
Then
$\nrank (\tilde A)$ 
is expected to exceed
the computed 
value $\rho$
if this stage outputs
 FAILURE and to equal $\rho$
otherwise, in which case 
Proto-Algorithm \ref{algbasmp} also
outputs a
rank-$\rho$ approximation of the matrix $\tilde A$
(within an error norm $\tau ||\tilde A||$
in $O(\sigma_{\rho+1}(\tilde A))$). 
For a sufficiently small tolerance  $\tau$
the latter outcome
implies that
certainly  $\rho\ge \nrank (\tilde A)$.
\end{remark}

\begin{remark}\label{rescond}
We can avoid 
pivoting and orthogonalization
in Subroutine COND 
by applying the Power 
or Lanczos algorithms \cite{B74}, \cite{D83},
\cite{KW92}.
We can first apply the Power Method
to the matrix $S= \tilde A^T \tilde A$ 
or $S= \tilde A \tilde A^T$
to yield  a 
  close  upper bound $\sigma_+^2$ 
on  its largest eigenvalue $\sigma_1^2(\tilde A)$
and then to the matrix
$\sigma_+^2 I- \tilde A^T \tilde A$
to approximate the smallest
eigenvalue of the matrix $S$,
equal to $\sigma^2_n( \tilde A)$.
The Lanczos algorithm
approximates 
both extremal eigenvalues 
of the matrix $S= \tilde A^T \tilde A$
(equal to $\sigma_1^2(\tilde A)$
and $\sigma_l^2(\tilde A)$, 
$l=\min\{m,n\}$
respectively) and
 converges much faster \cite[Sections 9.1.4, 9.1.5]{GL96}.
\end{remark}

\begin{remark}\label{resmpl1}
One can simplify Stage 2 by applying the Subroutine 
COND to the matrix $G'_{\rho}=F_{\rho}G_{\rho}$ of 
a smaller size (rather than to $G_{\rho}$)
where $F_{\rho}\in \mathcal G_{0,1}^{\rho\times m}$. By virtue of  
Theorem \ref{1} the matrices $G_{\rho}$ and $G'_{\rho}$ are
likely to have condition numbers of the same order.
\end{remark}

\begin{remark}\label{rebsch}
The binary search
$\rho=\lceil(\rho_-+\rho_+)/2\rceil$
is an attractive policy for choosing
the candidate values $\rho$,
but one may prefer to move 
toward $\rho_-$, the left end of the range 
more rapidly,
to decrease the size of the matrix $B'$.
\end{remark}


{\bf Acknowledgements:}
Our research has been supported by NSF Grant CCF--1116736 and
PSC CUNY Awards 64512--0042 and 65792--0043.
 We are also grateful to 
Mr. Jesse Wolf for helpful comments.






\begin{thebibliography}{hspace{0.5in}}


\bibitem[A94]{A94}
O. Axelsson, {\em Iterative Solution Methods},
Cambridge Univ. Press, 
England, 1994.


\bibitem[B74]{B74}
D.W. Boyd,
The Power method for $l^p$ Norms,
{\em Linear Algebra and Its Applications}, 
{\bf 9}, 95--101, 1974.


\bibitem[B80]{B80} 
D. Bini,
Border Rank of $p\times q\times 2$ Tensors and the Optimal Approximation 
of a Pair of  Bilinear Forms, in {\em Lecture Notes in Computer Science},
{\bf 85}, 98--108, Springer, 1980.
 

\bibitem[B86]{B86} 
D. Bini,
Border Rank of $m\times n\times (mn-q)$ Tensors,
{\em  Linear Algebra and Its Applications}, {\bf 79}, 45--51, 1986. 




\bibitem[B02]{B02}
M. Benzi, Preconditioning Techniques for Large Linear Systems: a Survey,
{\em J. of Computational Physics}, {\bf 182}, 418--477, 2002.


\bibitem[B11]{B11}
C. Beltr\'an, 
Estimates on the Condition Number of Random, Rank-deficient Matrices,
{\em IMA Journal of Numerical Analysis}, {\bf 31,~1}, 25--39, 2011. 


\bibitem[BM01]{BM01}
D. A. Bini, B. Meini, Approximate Displacement Rank and Applications,
in {\em AMS Conference "Structured Matrices in Operator Theory, Control,
Signal and Image Processing"}, Boulder, 1999 (edited by V. Olshevsky),
{\em American Math. Society}, 215--232, Providence, RI, 2001.


\bibitem[CD05]{CD05}
Z. Chen, J. J. Dongarra, Condition Numbers of Gaussian Random Matrices,
{\em SIAM. J. on Matrix Analysis and Applications}, {\bf 27}, 603--620, 2005.






\bibitem[D83]{D83}
J. D. Dixon, Estimating Extremal Eigenvalues and Condition Numbers of Matrices,
{\em SIAM J. on Numerical Analysis}, {\bf 20,~4}, 812--814, 1983.


\bibitem[D88]{D88}
J. Demmel,  
The Probability That a Numerical Analysis Problem Is Difficult,
{\em Math. of Computation}, {\bf 50}, 449--480, 1988.


\bibitem[DL78]{DL78}
R. A. Demillo, R. J. Lipton,
A Probabilistic Remark on Algebraic Program Testing,
{\em Information Processing Letters}, {\bf 7}, {\bf 4}, 193--195, 1978. 


\bibitem[DS01]{DS01}
K. R. Davidson, S. J. Szarek, 
Local Operator Theory, Random Matrices, and Banach Spaces, 
in {\em Handbook on the Geometry of  Banach Spaces} 
(W. B. Johnson and J. Lindenstrauss editors), pages 317--368, 
North Holland, Amsterdam, 2001. 


\bibitem[E88]{E88}
A. Edelman, Eigenvalues and Condition Numbers of Random Matrices,
{\em SIAM J. on Matrix Analysis and Applications}, {\bf 9}, {\bf 4},
543--560, 1988.


\bibitem[ES05]{ES05}
A. Edelman, B. D. Sutton,  Tails of Condition Number Distributions,
{\em SIAM J. on Matrix Analysis and Applications}, {\bf 27}, {\bf 2},
547--560, 2005.




\bibitem[G97]{G97}
A. Greenbaum,
{\em Iterative Methods  for Solving Linear Systems},
SIAM, Philadelphia, 
1997.






\bibitem[GE96]{GE96}
M. Gu, S. C. Eisenstat,
Efficient algorithms for computing a strong rank-revealing QR factorization,
{\em SIAM Journal on Scientific Computing}, {\bf 17}, 848--869, 1996.


\bibitem[GL96]{GL96}
G. H. Golub, C. F. Van Loan,
{\em Matrix Computations},
Johns Hopkins University Press, Baltimore, Maryland, 1996 (third addition).






\bibitem[H02]{H02}
N. J. Higham, {\em Accuracy and Stability in Numerical Analysis},
SIAM, Philadelphia, 2002 (second edition).


\bibitem[HMT11]{HMT11}
N. Halko, P. G. Martinsson, J. A. Tropp,
Finding Structure with Randomness: Probabilistic Algorithms
for Constructing Approximate Matrix Decompositions, 
{\em SIAM Review}, {\bf 53,~2}, 217--288, 2011.


\bibitem[HP92]{HP92}
Y. P. Hong, C.--T. Pan,
 The rank revealing QR decomposition and SVD, 
{\em Math. of Computation}, {\bf 58}, 213--232, 1992.




\bibitem[KB09]{KB09}
T. G. Kolda, B. W. Bader, 
Tensor Decompositions and Applications,
{\em SIAM Review}, {\bf 51,~3}, 455--500, 2009.


\bibitem[KKM79]{KKM79}
T. Kailath, S. Y. Kung, M. Morf,
Displacement Ranks of Matrices and Linear Equations,
{\em Journal Math. Analysis and Appls}, {\bf 68,~2}, 395--407, 1979.


\bibitem[KL94]{KL94}
C. S. Kenney, A. J. Laub,
Small-Sample Statistical Condition Estimates for General Matrix Functions, 
{\em SIAM J. on Scientific and Statistical Computing}, {\bf 15}, 36--61, 1994.


\bibitem[KW92]{KW92}
J. Kuczynski, H. Wozniakowski,
Estimating the Largest Eigenvalue by the Power and Lanczos
Algorithms with a Random Start,
{\em SIAM J. on Matrix Analysis and Applications}, {\bf 13}, 1094--1122, 1992.



\bibitem[O09]{O09} 
I. V. Oseledets,
 A Compact Matrix Form of the d-Dimensional Tensor Decomposition,
Preprint 2009-01, {\em INM RAS}, March 2009.




\bibitem[OT09]{OT09} 
I. V. Oseledets, E. E. Tyrtyshnikov,
Breaking the Curse of Dimensionality, or How to Use SVD in Many Dimensions,
{\em SIAM J. Scientific Comp.}, {\bf 31}, {\bf 5}, 3744--3759, 2009.


\bibitem[OT10]{OT10} 
I. V. Oseledets, E. E. Tyrtyshnikov,
TT-cross Approximation for Multidimensional Arrays,
{\em Linear Algebra Appls.} {\bf 432,~1}, 70--88, 2010.


\bibitem[OT11]{OT11} 
I. Oseledets, E. E. Tyrtyshnikov, 
Algebraic Wavelet Transform via Quantics Tensor Train Decomposition,
{\em SIAM J. Sci. Comp.}, {\bf 33}, {\bf 3},  
1315--1328, 2011.


\bibitem[P72]{P72}
V. Y. Pan,
On Schemes for the Evaluation of Products and Inverses of Matrices (in Russian), 
{\em Uspekhi Matematicheskikh Nauk}, {\bf 27}, {\bf 5} {\bf (167)}, 249--250, 1972.


\bibitem[P79]{P79}
V. Y. Pan,
Fields Extension and Trilinear Aggregating, 
Uniting and Canceling for the Acceleration of Matrix Multiplication, 
{\em Proceedings of the 20th Annual IEEE Symposium on Foundations of Computer Science (FOCS '79)},
28--38, IEEE Computer Society Press, Long Beach, California, 1979. 


\bibitem[P84]{P84}
V. Y. Pan, How Can We Speed up Matrix Multiplication?
{\em SIAM Review,} {\bf 26}, {\bf 3}, 393--415, 1984.


\bibitem[P00a]{P00a}
C.--T. Pan,
On the Existence and Computation
of Rank-revealing LU Factorization,
{\em Linear Algebra and Its Applications}, {\bf 316}, 199--222, 2000.


\bibitem[P01]{p01}
V. Y. Pan,
{\em Structured Matrices and Polynomials: Unified Superfast Algorithms},
Birkh\"auser/Springer, Boston/New York, 2001.


\bibitem[P10]{P10}
V. Y. Pan,
Newton's Iteration for Matrix Inversion, Advances and Extensions, 
pp. 364--381, in {\em Matrix Methods: Theory, Algorithms and Applications} 
(dedicated to the Memory of Gene Golub, edited by V. Olshevsky and E. Tyrtyshnikov), 
World Scientific Publishing, New Jersey, ISBN-13 978-981-283-601-4, ISBN-10-981-283-601-2 (2010).


\bibitem[PGMQ]{PGMQ}
V. Y. Pan, D. Grady, B. Murphy, G. Qian, R. E. Rosholt, A. Ruslanov,
Schur Aggregation for Linear Systems and Determinants,
{\em Theoretical Computer Science}, 
{\em Special Issue on Symbolic--Numerical Algorithms}
(D. A. Bini, V. Y. Pan, and J. Verschelde editors), 
{\bf 409}, {\bf 2}, 255--268, 2008.


\bibitem[PIMR10]{PIMR10} 
V. Y. Pan, D. Ivolgin, B. Murphy, R. E. Rosholt, Y. Tang, X. Yan,
Additive Preconditioning for Matrix Computations,
{\em Linear Algebra and Its Applications}, {\bf 432}, 1070--1089, 2010.


\bibitem[PQ10]{PQ10}
V. Y. Pan, G. Qian,
Randomized Preprocessing of Homogeneous Linear Systems of Equations, 
{\em Linear Algebra and Its Applications}, {\bf 432}, 3272--3318, 2010.


\bibitem[PQ12]{PQ12}
V. Y. Pan, G. Qian,
Solving Linear Systems of Equations 
with Randomization, Augmentation and Aggregation,
{\em Linear Algebra and Its Applications}, {\bf 437},
2851--1876, 2012.


\bibitem[PQa]{PQa}
V. Y. Pan, G. Qian, 
Condition Numbers of Random Toeplitz and Circulant Matrices,
Tech. Report TR 2012013,
\textit{PhD Program in Comp. Sci.}, \textit{Graduate Center, CUNY}, 2012.

Available at http://www.cs.gc.cuny.edu/tr/techreport.php?id=442


\bibitem[PQZa]{PQZa}
V. Y. Pan, G. Qian, A. Zheng,
Randomized Preprocessing versus Pivoting, 
{\em Linear Algebra and Its Applications}, in print. 
http://dx.doi.org/10.1016/j.laa.2011.02.052


\bibitem[PQZb]{PQZb}
V. Y. Pan, G. Qian, A. Zheng,
Randomized  Matrix Computations,
Tech. Report TR 2012009,
\textit{PhD Program in Comp. Sci.}, \textit{Graduate Center, CUNY}

Available at http://www.cs.gc.cuny.edu/tr/techreport.php?id=438
and 

http://arxiv.org/abs/1210.7476


\bibitem[PQZc]{PQZc}
V. Y. Pan, G. Qian, A. Zheng,
Randomized  Augmentation and Additive Preprocessing,
Tech. Report TR 201201x,
\textit{PhD Program in Comp. Sci.}, \textit{Graduate Center, CUNY}
 2012.


\bibitem[PQZC]{PQZC}
V. Y. Pan, G. Qian, A. Zheng, Z. Chen,
Matrix Computations and Polynomial Root-finding with Preprocessing, 
{\em Linear Algebra and Its Applications}, {\bf 434}, 854--879, 2011.


\bibitem[PY09]{PY09}
V. Y. Pan, X. Yan,
Additive Preconditioning, Eigenspaces, and the Inverse Iteration, 
{\em Linear Algebra and Its Applications}, {\bf 430}, 186--203, 2009.


\bibitem[S80]{S80}
J. T. Schwartz,
Fast Probabilistic Algorithms for Verification of Polynomial Identities, 
{\em Journal of ACM}, {\bf 27}, {\bf 4}, 701--717, 1980. 


\bibitem[S98]{S98}
G. W. Stewart,
{\em Matrix Algorithms, Vol I: Basic Decompositions},
SIAM, 
1998.


\bibitem[SST06]{SST06}
A. Sankar, D. Spielman, S.-H. Teng, 
Smoothed Analysis of the Condition Numbers and Growth Factors of Matrices, 
{\em SIAM J. on Matrix Analysis}, {\bf 28}, {\bf 2}, 446--476, 2006. 


\bibitem[T11]{T11}
J. A. Tropp, Improved analysis of the subsampled randomized Hadamard transform,
{\em Adv. Adapt. Data Anal.}, {\bf 3,~1--2}, 
Special Issue, "Sparse Representation of Data and Images,"  
115-126, 2011.


\bibitem[Z79]{Z79}
R. E. Zippel, Probabilistic Algorithms for Sparse Polynomials, 
{\em Proceedings of EUROSAM'79, Lecture Notes in Computer Science}, 
{\bf 72}, 216--226, Springer, Berlin, 1979.


\end{thebibliography}
\end{document}